\documentclass[]{article}

\usepackage[T1]{fontenc}
\usepackage[utf8]{inputenc}
\usepackage{lmodern}
\usepackage[
    a4paper, 
    top=3cm, 
    bottom=3cm, 
    left=3cm, 
    right=3cm]{geometry}
\usepackage{titling}
\pretitle{\begin{center}\LARGE\bfseries}
\posttitle{\par\end{center}}
\usepackage[english]{babel}
\usepackage{csquotes}
\usepackage{amsmath, amssymb}
\usepackage{amsthm}
\usepackage{xcolor}
\usepackage{mathtools}
\usepackage{graphicx}
\usepackage{array}
\usepackage{titlesec}
\usepackage{eepic}
\usepackage{multicol}
\usepackage[dvipsnames]{xcolor}
\usepackage{color}
\usepackage[hypcap=true]{caption}
\usepackage[cmtip,all]{xy}
\usepackage{enumitem}
\usepackage{leftidx}
\usepackage[mathscr]{euscript}
\usepackage{setspace}
\usepackage{needspace}
\usepackage{tikz}
\usetikzlibrary{cd}

\usepackage[
  backend=biber,
  style=alphabetic,
  sorting=nyt,
  giveninits=true,
  maxbibnames=99,
  maxalphanames=3,
  minalphanames=3,
  date=year,
  doi=false,
  url=false,
  eprint=false,
  isbn=false
]{biblatex}

\usepackage[hidelinks, linktocpage=true]{hyperref}
\hypersetup{
    colorlinks = true,
    citecolor  = blue!75!black,  
    urlcolor   = blue!75!black,   
    linkcolor  = red!70!black 
}

\usepackage[linktocpage=true,hidelinks]{hyperref}

\DeclareNameAlias{author}{given-family}
\DeclareNameAlias{editor}{given-family}
\DeclareNameAlias{translator}{given-family}

\DeclareDelimFormat[bib]{nametitledelim}{\addcomma\space}

\DeclareFieldFormat[article,inproceedings,unpublished]{title}{#1}

\DeclareFieldFormat[incollection]{title}{#1}
\DeclareFieldFormat[misc]{title}{#1\isdot}

\DeclareFieldFormat[article]{journaltitle}{#1}

\DeclareFieldFormat[article]{volume}{\mkbibbold{#1}}

\DeclareFieldFormat[article,inproceedings]{pages}{#1}

\renewbibmacro*{in:}{%
  \ifentrytype{article}
    {}
    {\bibstring{in}\printunit{\intitlepunct}}}

\renewbibmacro*{journal+issuetitle}{%
  \usebibmacro{journal}%
  \setunit*{\addspace}%
  \printfield{volume}%
  \setunit{\addspace}%
  \printtext[parens]{\printdate}%
  \newunit}

\renewbibmacro*{note+pages}{%
  \printfield{note}%
  \setunit{\addcomma\space}%
  \iffieldundef{pages}
    {\printfield{eid}}
    {\printfield{pages}}%
  \newunit}

\DeclareLabelalphaTemplate{
  \labelelement{
    \field[final]{shorthand}
    \field{label}
    \field[strwidth=2,strside=left,ifnames=1]{labelname}
    \field[strwidth=1,strside=left]{labelname}
  }
  \labelelement{
    \field[strwidth=2,strside=right]{year}
  }
}

\AtEveryBibitem{%
  \ifentrytype{inproceedings}
    {%
      \clearlist{publisher}%
      \clearlist{location}%
    }
    {}%
}

\AtEveryBibitem{%
  \ifentrytype{misc}
    {\clearfield{eprintclass}}
    {}%
}

\usepackage{parskip}

\makeatletter
\def\thm@space@setup{%
  \thm@preskip=\parskip \thm@postskip=0pt
}
\makeatother

\allowdisplaybreaks

\usepackage{titlesec}
\titleformat{\section}[block]{\color{black}\large\bfseries\filcenter}{\thesection.}{0.5em}{}
\titleformat{\subsection}[hang]{\bfseries}{}{0.5em}{}

\numberwithin{equation}{section}

\newtheorem{theorem}{Theorem}[section]
\newtheorem{lemma}[theorem]{Lemma}
\newtheorem{proposition}[theorem]{Proposition}

\newtheorem{problem}[theorem]{Problem}

\newtheorem{remark}[theorem]{Remark}

\newtheorem{ltheorem}{Theorem}

\newtheorem{Btheorem}{Theorem}

\titleformat{\subsection}[runin]
  {\bfseries}{\thesubsection.}{0.5em}{}[.]

\titleformat{\subsubsection}[runin]
  {\bfseries}{}{0pt}{}[.]

\makeatletter
\renewenvironment{proof}[1][\proofname]{%
   \par\pushQED{\qed}\normalfont%
   \topsep6\p@\@plus6\p@\relax
   \trivlist\item[\hskip\labelsep\bfseries#1\@addpunct{.}]%
   \ignorespaces
}{%
   \popQED\endtrivlist\@endpefalse
}
\makeatother

\newcommand{\CC}{\mathbf{C}}
\newcommand{\DD}{\mathbf{D}}
\newcommand{\EE}{\mathbf{E}}
\newcommand{\FF}{\mathbf{F}}

\newcommand{\HH}{\mathbf{H}}

\newcommand{\NN}{\mathbf{N}}

\newcommand{\QQ}{\mathbf{Q}}
\newcommand{\RR}{\mathbf{R}}
\newcommand{\TT}{\mathbf{T}}

\newcommand{\ZZ}{\mathbf{Z}}

\newcommand{\A}{\mathcal{A}}
\newcommand{\B}{\mathcal{B}}
\newcommand{\C}{\mathcal{C}}

\newcommand{\G}{\mathcal{G}}

\def\L{\mathcal{L}}
\newcommand{\M}{\mathcal{M}}

\newcommand{\R}{\mathcal{R}}
\newcommand{\T}{\mathcal{T}}

\newcommand{\V}{\mathcal{V}}

\newcommand{\Cc}{\mathscr{Cs}}

\newcommand{\sA}{\mathrm{A}}
\newcommand{\sB}{\mathrm{B}}

\newcommand{\sG}{\mathrm{G}}

\newcommand{\sK}{\mathrm{K}}

\newcommand{\sN}{\mathrm{N}}

\newcommand{\into}{\hookrightarrow}

\DeclareMathOperator*\supp{\mathrm{supp}}

\def\XXint#1#2#3{{\setbox0=\hbox{$#1{#2#3}{\int}$ }
\vcenter{\hbox{$#2#3$ }}\kern-.6\wd0}}

\def\1{\mathbf{1}}
\def\Id{\mathrm{id}}

\def\M{\mathcal{M}}
\def\A{\mathcal{A}}

\def\Fou{\mathcal{F}}

\renewcommand{\Re}{\operatorname{\mathbf{Re}}}
\renewcommand{\Im}{\operatorname{\mathbf{Im}}}

\newcommand{\rad}{\mathrm{rad}}
\newcommand{\vertiii}[1]{
 {\left\vert\kern-0.25ex\left\vert\kern-0.25ex\left\vert #1 
  \right\vert\kern-0.25ex\right\vert\kern-0.25ex\right\vert}
}

\def\CB{\mathcal{C B}}

\def\cb{\mathrm{cb}}

\def\Tr{\mathrm{Tr}}

\def\diag{\mathrm{diag}}
\newcommand{\ev}{\mathrm{ev}}

\def\SL{\mathrm{SL}}

\def\PGL{\mathrm{PGL}}

\def\PSL{\mathrm{PSL}}

\def\PSO{\mathrm{PSO}}
\def\HB{\mathrm{O}} 

\newcommand{\acts}{\curvearrowright}
\newcommand{\Sph}{\Theta}

\title{ Radial Schur multipliers on trees\\[1pt] and the failure of Fourier-Schur transference}
\author{Adrián M. González-Pérez, Javier Parcet, \\Jorge Pérez García and Simeng Wang}
\date{}

\begin{document}

\maketitle

\begin{abstract}
Given $1 < p < \infty$ and a function $m \colon \NN_0 \to \CC$, we provide a characterization of radial Schur multipliers $M(x,y) = m(d(x,y))$ on the Schatten $p$-class of a tree $(X,d)$ equipped with its graph metric. That is, the operator norm of the radial Schur multiplier $S_M$ is controlled by the norm of the Fourier multiplier $T_m$ on the torus:
\[
    \big\| S_M \colon S_p(X) \to S_p(X) \big\|_\cb \lesssim \frac{p^4}{(p-1)^2} \, \big\| T_m \colon L_p(\TT) \to L_p(\TT) \big\|_\cb,
\]
with both quantities being comparable when the tree $X$ has infinite diameter. This is a rather unexpected inequality, which necessarily fails at the endpoints $p=1,\infty$ due to the work of Haagerup, Steenstrup and Szwarc. As an application, we construct the first examples of completely $S_p$-bounded Herz-Schur multipliers in the free group that are unbounded as Fourier multipliers. This result disproves Fourier-Schur transference---a longstanding open problem since the work of Neuwirth and Ricard in 2011. In line with this, our main result also implies the surprising fact that free balls Schur multipliers are uniformly $S_p$-bounded, in sharp contrast to their Fourier behavior. In addition, we disprove Fourier-Schur transference over the algebraic groups $\PGL_2(\QQ_q)$ and $\PSL_2(\RR)$ with $\sK$-biinvariant counterexamples, the latter involving additional geometric and probabilistic techniques.
\end{abstract}

\section*{Introduction}

The theory of Schur multipliers on Schatten $p$-classes plays a key role in several aspects of harmonic analysis and geometric group theory. Before exploring these connections, let us formulate our main result. Given a discrete set $\Omega$ and $M \colon \Omega \times \Omega \to \CC$ bounded, the corresponding Schur multiplier acts on matrices by entrywise multiplication
\[
    S_M(A) = \Big[ M(j,k) A_{j,k} \Big]_{j,k \in \Omega}.
\]
Throughout the text, the operator $S_M$ will act on the Schatten $p$-classes $S_p(\Omega)$. As usual, we will also endow $S_p$ with its natural operator space structure and denote by $\| \cdot\|_\cb$ the completely bounded norm of a linear operator, see Section \ref{Sect1} for the precise definitions.
This article starts with Schur multipliers on trees $(X,d)$ equipped with their graph distance $d$. Indeed, a function $M \colon X \times X \to \CC$ will be called radial if it is of the form $M(x,y) = m(d(x,y))$ for some bounded function $m \colon \NN_0 \to \CC$, which we identify with its even extension $m \colon \ZZ \to \CC$. Our main result gives a characterization of radial $S_p$-bounded Schur multipliers on trees.


\begin{ltheorem}
    \label{Main Thm}
    Let $M(x,y) = m(d(x,y))$ be a radial symbol on a tree and $T_m$ the corresponding even Fourier multiplier on the torus $\TT$. Then, for every $1 < p < \infty$ we have 
    \[
        \big\| S_M \colon S_p(X) \to S_p(X) \big\|_\cb \, \lesssim \, p_\ast^2 \, \big\| T_m \colon L_p(\TT) \to L_p(\TT) \big\|_\cb,
    \]
    where $p_\ast := \max\{p,\tfrac{p}{p-1}\}$. If $X$ has infinite diameter, the reciprocal inequality also holds 
    \[
        \big\| S_M \colon S_p(X) \to S_p(X) \big\|_\cb \ge \big\| T_m \colon L_p(\TT) \to L_p(\TT) \big\|_\cb.
    \]
\end{ltheorem}
 
In other words, in the completely bounded setting, radial Schur multipliers on unbounded trees are just even Fourier multipliers on the torus determined by their radial symbol. This clean statement was certainly unexpected for the way it fits in the representation theory of free groups and for its implications in harmonic analysis. The paper is organized around the following results.

\textbf{A. Spherical functions and Schur multipliers.} 
Given a tree $X$, in which a root $o \in X$ is fixed, we will say that a function $\varphi: X \to \CC$ is radial if it depends only on its distance to the root $o \in X$. The spherical functions of $X$ are those radial functions that are also formal eigenvectors of the normalized adjacency operator, see \cite[Definition 2.4]{HaagSteenSw2009}. When $X$ is homogeneous, as its degree increases, its spherical functions converge to $x \mapsto z^{d(o,x)}$, for $z \in \CC$. As such, radial symbols $z^{d(x,y)}$ can be understood as spherical functions for trees of infinite degree. Our main technical inequality gives the following uniform bound
\begin{equation} 
    \tag{Sph$_\TT$}
    \label{SphericalInequality}
    \sup_{|z| = 1} \Big\| \big[ z^{d(x,y)} A_{x,y} \big]_{x,y} \, \Big\|_{S_p(X)} \lesssim p_\ast \, \|A\|_{S_p(X)} \quad \mbox{for all} \quad 1 < p < \infty.
\end{equation}

In contrast with our result above, Pytlik/Szwarc proved in their celebrated paper \cite{PytlikSzwarc1986} that the $S_\infty$-analog of \eqref{SphericalInequality} holds for the free group with infinitely many generators iff $|z| < 1$ or $z = \pm 1$. They obtained this by expressing the spherical functions $g \mapsto z^{|g|}$ on $\FF_\infty$ as matrix coefficients 
\[
z^{|g|} = \big\langle \pi_z(g) \delta_e, \delta_e \big\rangle
\]
of an analytic family of bounded representations.  We refer to Haagerup/Steenstrup/Szwarc \cite{HaagSteenSw2009} for a full characterization of radial Schur multipliers on $S_\infty$ valid over homogeneous trees of finite degree.
%
Our result in $S_p$ is a counterpart to their results in $S_\infty$ that exhibits a striking change of regime at the boundary $|z|=1$. In line with this \textit{change of regime}, we see that radial $S_p$ multipliers on a tree coincide with $L_p$ multipliers on the torus for $1<p<\infty$, while for $p=\infty$ they form a proper subset of the multipliers on $\TT$.
The argument for \eqref{SphericalInequality} is completely new. It consists of reducing it to an instance of Burkholder's inequality for noncommutative martingale transforms \cite{PisierXu1997Martingales, Randrianantoanina2002, ParcetRandrianantoanina2006}. The link comes from a certain filtration---already used in \cite{MixedNorm2026} to deal with mixed-norm inequalities for free ball multipliers---which gives the fundamental identity below for the phase associated to the Gromov form
\[
    \big[ z^{|x \wedge y|} A_{x,y} \Big]_{x,y \in X} = \mathrm{diag} \big( z^{|x|} A_{x,x} \big) + \sum_{k\geq 0} z^{k} \mathbf{d}_{k} (A),
\]
where $\mathbf{d}_{k} = \EE_k - \EE_{k+1}$ denotes the martingale difference associated with the filtration, $x \wedge y$ is the last common vertex of the geodesics from the root $o$ to $x$ and $y$ respectively and $|x| = d(o,x)$ denotes the distance to said root.

Once \eqref{SphericalInequality} is obtained, Theorem \ref{Main Thm} follows by a natural intertwining identity between Schur and Fourier multipliers, see \eqref{eq:Intertwining}. None of these identities use the homogeneity of $X$ and thus they apply to general trees. However, our main theorem easily implies more technical generalizations of \eqref{SphericalInequality} where $X$ is a $(q+1)$-homogeneous tree, the functions $z^n$ are replaced by the spherical functions on said tree $\varphi_z$ and $z$ lives in the boundary of the ellipse described in \cite{HaagSteenSw2009}, see Theorem \ref{thmEllipsoid} and the preceding discussion for the details. 

\textbf{B. The failure of Fourier-Schur transference over free groups.} 
%
%
Schur multipliers are intimately connected with several aspects of harmonic analysis, among which we can highlight two. First, the pioneering work of Haagerup on free groups \cite{Haa1978} and subsequent research on rank 1 lattices \cite{CarHaa1985multipliers, CowHaa1989, Haagerup2016NonCBAP} strongly relied on Fourier and Schur $L_\infty$-approximations to produce analytic invariants of their group von Neumann algebras. In the development of this program, Schur multipliers were deployed as the main technical tool. Second, Krein's problem on operator Lipschitz functions was settled in \cite{PotapovSukochev2011} by solving Arazy's conjecture on Schur multipliers of divided differences---the first purely noncommutative models of singular integral operators. 

These research lines illustrate the large and growing impact of Schur multipliers in harmonic analysis and operator algebras---we refer to \cite{Parcet2026ICM} for recent developments and historical insights. Many of these rely, in one way or another, on the strong relation between Fourier and Schur multipliers. Indeed, given a bounded function on a discrete group $m: \Gamma \to \CC$, we will denote by $T_m$ its associated Fourier multiplier acting on the left regular von Neumann algebra $\L \Gamma$ and its noncommutative $L_p$-spaces. Similarly, its associated Herz-Schur multiplier $H_m$ is the Schur multiplier 
\[
    H_m := S_M, \, \text{ where } \, M(\gamma,\eta) = m(\gamma \eta^{-1}). 
\]
The definitions of both Fourier and Herz-Schur multipliers generalize to locally compact groups \cite{LaffSall2011, CasSall2015}. Since the work of Bo\.zejko/Fendler \cite{BozejkoFendler1984HSFourier} it has been known that the boundeness of Fourier and Herz-Schur multipliers coincides for general groups when $p=\infty$. In contrast, the situation is much more delicate for $p \in (1,\infty) \setminus \{2\}$. There, the argument of \cite{BozejkoFendler1984HSFourier} breaks down. However, in 2011 Neuwirth/Ricard \cite{NeuRic2011} showed that the same result holds true for amenable discrete groups. Indeed, their result claims that for any amenable discrete group $\Gamma$ and $1\leq p \leq \infty$ one has 
\begin{equation} 
    \tag{FS$_p$} 
    \label{FSTransference}
    \big\| H_m \colon S_p(\Gamma) \to S_p(\Gamma) \big\|_\cb = \big\| T_m \colon L_p(\mathcal{L}\Gamma) \to L_p(\mathcal{L}\Gamma) \big\|_\cb.
\end{equation}
That is, completely $S_p$-bounded Herz-Schur multipliers on $\Gamma$ can be isometrically identified with completely $L_p$-bounded Fourier multipliers on its group von Neumann algebra. This fundamental result was extended to l.c. groups by Caspers/de la Salle's \cite{CasSall2015}, but it left open whether the amenability assumption in \eqref{FSTransference} was necessary or just an artifact of the known proofs. The upper bound ($\leq$) is known to hold for nonamenable l.c. groups, see \cite[Theorem 4.2]{CasSall2015}, but the lower bound, that is the control of the Fourier multiplier norm by the Herz-Schur one, has resisted all attempts until now. Our counterexamples below imply that \eqref{FSTransference} fails if $\Gamma$ has a free subgroup. However, whether \eqref{FSTransference} implies amenability remains open.  

Although the general lower-bound problem for \eqref{FSTransference} on nonamenable groups has remained open for the last 15 years, there has been some progress on related questions. For instance, a local form of the lower bound for nonamenable groups was developed in \cite{ParcetSLn2018} and \cite{Parcet2024Idempotentschur}, with profound consequences in harmonic analysis on Lie groups \cite{Parcet2024Idempotentschur, CaspersJanssensKrishnaswamyUshaMiaskiwskyi2024}. This local form of Fourier-Schur transference is trivial for finite sets on discrete groups, although the precise bounds seem like a challenging quantitative problem outlined in Subsection \ref{SectFSTransfRatio} below. 

We start by showing the failure of Fourier-Schur transference on free groups with a couple of fairly natural examples. First, let $\FF_r$ be the free group on $r$ generators. Recall that any radial function on the free group induces a radial Herz-Schur multiplier on its Cayley graph with respect to the word length metric.
%
%
Write $B_R(w)$ for the ball of radius $R>0$ centered at $w \in \FF_r$. Then, the \emph{free ball multipliers} are defined by the symbol $m_R = \1_{B_R(e)}$. Next, pick $\rho \in \partial \FF_r$ in the hyperbolic boundary of $\FF_r$ and identify it with the geodesic path $(\rho_n)_{n \geq 1}$ starting at $e \in \FF_r$. The \emph{free horoball} $\HB_\rho \subset \FF_r$ associated to $\rho$ is given by 
\[
    \HB_\rho = \bigcup_{n \geq 0} B_n(\rho_n).
\]
Let $m_\rho = \1_{\HB_\rho}$ be the symbol given by the characteristic function of the free horoball. 
We will show in Lemma \ref{HoroballvsBall} that the uniform boundedness of the free ball multipliers on $L_p(\L \FF_r)$ is equivalent to the boundedness of the free horoball multiplier $T_{m_\rho}$. In sum, we get the following.

\begin{Btheorem}\label{CorFree}
    Consider the free group $\FF_r$ on $r$ generators with $r \in \{2, 3, \dots, \infty\}$:
    \begin{enumerate}[label={\rm \textbf{(\roman*)}}, ref={(\roman*)}]
        \item \label{itm:CorFree1}
        \textbf{Free balls.} The Herz-Schur multipliers $H_{m_R}$ are uniformly cb-bounded
        \[
            \sup_{R>0} 
            \big\| H_{m_R} \colon S_p(\FF_r) \to S_p(\FF_r) \big\|_\cb 
            \lesssim p_\ast^3 \quad \text{for every} \quad 1 < p < \infty.
        \]
        However, the Fourier multipliers $T_{m_R} \colon L_p(\L \FF_r) \to L_p(\L \FF_r)$ are not uniformly bounded when $|\frac1p - \frac12| > \frac16$.
        \item \label{itm:CorFree2}
        \textbf{Free horoballs.} Given $\rho \in \partial \FF_r$, the Herz-Schur multiplier
        $H_{m_\rho} \colon S_p(\FF_r) \to S_p(\FF_r)$ is cb-bounded. However, $T_{m_\rho} \colon L_p(\mathcal{L}\FF_r) \to L_p(\mathcal{L}\FF_r)$ is unbounded for $|\frac1p - \frac12| > \frac16$. 
    \end{enumerate}
\end{Btheorem}
The above result establishes the uniform (complete) boundedness of free ball multipliers in the Schur context. This contrasts starkly with the Fourier setting \cite{Kuhn1984Convergence, BoFen2006,} for $|\frac1p - \frac12| > \frac16$. In fact, free ball Fourier multipliers are expected to be non-uniformly bounded for every $p \neq 2$ ---in line with Fefferman's ball multiplier theorem \cite{Fef1971}--- but this remains an open problem related to difficulties of applying \emph{Fefferman's Kakeya obstruction} in the noncommutative setting, see \cite{Parcet2024Idempotentschur}. More information will be provided in Remark \ref{Rem-Curvature}.

Theorem \ref{CorFree} is enough to establish that Fourier-Schur transference fails in the range $|\frac1p - \frac12| > \frac16$. However, due to our main theorem, explicit radial counterexamples are now easy to manufacture. Consider the spherical functions in $\FF_\infty$, given by 
\[
    m_z(g) = z^{|g|} \; \text{ for } z \in \TT
\] 
and their associated Herz-Schur multipliers, which are just the radial Schur multipliers in \eqref{SphericalInequality}. This family provides radial counterexamples for every nontrivial exponent $p$. 

\begin{Btheorem}
    \label{corFree2}
    Let $z \in \TT \setminus \{\pm 1\}$. For any $1 < p \neq 2 < \infty$ we have that $H_{m_z}$ is a completely bounded Herz-Schur multiplier while the Fourier multiplier $T_{m_z}$ is unbounded on $L_p(\L \FF_r)$. 
\end{Btheorem}

The Theorem above fully settles that \eqref{FSTransference} fails for discrete groups that contain a free subgroup. However, there are now very well known examples of non-amenable groups that do not contain a free group, see \cite{Olshanskii1980InvariantMean, OlshanskiiSapir2003TorsionByCyclic, Bartholdi2018Chapter}. Thus, the problem of determining whether the class of groups exhibiting Fourier-Schur transference contains just the amenable groups is still open.


\textbf{C. Algebraic groups and $\sK$-biinvariant multipliers} 
The work of Haagerup and coauthors on $L_\infty$-approximation theory \cite{Haa1978, CarHaa1985multipliers, CowHaa1989} has been continued by Lafforgue/de Laat/de la Salle \cite{LaffSall2011,LaatSalle2018Approximation} with the failure of $L_p$-approximations for high rank lattices. These more recent results have laid out a program that exploits harmonic analysis connections to attack the elusive rigidity properties of higher rank group von Neumann algebras \cite{ParcetSLn2018, Salle2022ICM, Parcet2024Idempotentschur, Salle2026Kakeya,Parcet2026ICM}.
They showed that, over higher rank algebraic simple groups, there aren't sequences of compactly supported Herz-Schur multipliers approximating the identity. This, together with Schur restriction properties and the upper bound ($\le$) in \eqref{FSTransference}, gives an obstruction to the CBAP in $L_p$. 

Whether the existence of approximations by Schur multipliers can be used to ensure the CBAP is an open problem that would be implied by the validity of Fourier-Schur transference. 
Over algebraic groups, approximations by compactly supported multipliers can always be averaged to be $\sK$-biinvariant. Thus, whether Fourier-Schur transference holds for $\sK$-biinvariant multipliers in this context became a question of utmost importance. In addition, the fact that the $\sK$-biinvariant functions form an Abelian algebra and that the algebraic groups involved are of type I further gave experts in the area hope of a positive solution. Our Theorem \ref{CorSimple} below provides $\sK$-biinvariant counterexamples in the easiest algebraic simple groups. Thus, if a relationship between Schur and Fourier approximation properties exists, it has to be more subtle, see Section \ref{Sect:Problems}.

Let us first recall some terminology related to the algebraic groups $\PGL_2(\QQ_q)$ for some prime number $q$ and $\PSL_2(\RR)$. The subgroup $\PGL_2(\ZZ_q)$, where $\ZZ_q$ is the ring of unit length elements, is maximally compact in $\PGL_2(\QQ_q)$ and the associated Cartan length $\ell_\mathrm{Car} \colon \PGL_2(\QQ_q) \to \RR_+$ is given by the relation $\ell_\mathrm{Car}(g)=n$, where $n \in \NN_0$ is uniquely determined by the Cartan decomposition 
\[
    g = 
    k_1 \, 
    \begin{bmatrix} 
        q^n & 0 \\ 
        0 & 1
    \end{bmatrix} 
    \, k_2 \quad \mbox{with} \quad k_1, \, k_2 \in \PGL_2(\ZZ_p).
\]

Next, consider the  lattice $\Gamma = \PSL_2(\ZZ) = \big\langle a,b \colon a^2 = b^3 = 1 \big\rangle = \ZZ_2 * \ZZ_3$ in $\sG = \PSL_2(\RR)$. Equip $\Gamma$ with the displacement length of the action on its Bass-Serre tree $T$, given by $\ell_T(\gamma) = d(o, \gamma \cdot o)$. Applying the arguments in Theorem \ref{Main Thm} to $T$ we get
\[
    m_z(\gamma) = z^{\ell_T(\gamma)}
\]
induces uniformly bounded Herz-Schur multipliers on $S_p(\Gamma)$ for $|z|= 1$. Now, let $\Omega = \Gamma \backslash \sG$ and denote its normalized $\sG$-invariant probability measure by $\mu$. Fix a measurable section $\sigma: \Omega \to \sG$ associated with the standard fundamental domain of the modular action, whose right cocycle $\beta \colon \Omega \times \sG \to \Gamma$ is determined by $\sigma(w) g = \beta(w,g) \sigma(w g)$. Then, the cocycle amplifications are given by 
\begin{equation*}
    \widehat{m}_z(g) = \int_\Omega m_z(\beta(w,g)) \, d\mu(w) \quad \mbox{for} \quad g \in \sG.
\end{equation*}
We know from \cite{WangXiaYao2025} that the symbols $\widehat{m}_z \colon \PSL_2(\RR) \to \RR_+$ are $\mathrm{PSO}(2)$-biinvariant continuous functions. Using a combination of techniques from probability and continuous fractions,  we will show that they fail to be uniformly bounded. Thus, there are $\sK$-biinvariant counterexamples to Fourier-Schur transference over simple nonamenable algebraic groups like $\PSL_2(\RR)$ and $\PGL_2(\QQ_q)$. 


\addtocounter{ltheorem}{1}
\begin{ltheorem}
    \label{CorSimple}
    The following results hold for $1 < p \neq 2 < \infty \hskip-1pt \colon$
    \begin{enumerate}[label={\rm \textbf{(\roman*)}}, ref={{(\roman*)}}]
        \item \label{itm:CorSimple1} 
        \textbf{Explicit multipliers in $\sG = \PGL_2(\QQ_q)$.} Let $z \in \TT \setminus \{\pm 1\}$ and $m_z(g) = z^{\ell_\mathrm{Car}(g)}$. Then $m_z$ is $\sK = \PGL_2(\ZZ_p)$-biinvariant, its Herz-Schur multiplier $H_{m_z}$is completely bounded on the Schatten $p$-class $S_p(\sG)$, while $T_{m_z}$ is unbounded on $L_p(\L \sG)$ as a Fourier multiplier. 

        \item \label{itm:CorSimple2}
        \textbf{Cocycle amplification in $\sG = \PSL_2(\RR)$ of phase shifts.} Let $z \in \TT$ and ${\widehat{m}_z}$ be the cocycle amplification above. Then $\widehat{m}_z$ is $\mathrm{PSO}(2)$-biinvariant, the family of Herz-Schur multipliers $H_{\widehat{m}_z}$ are uniformly completely bounded on $S_p(\sG)$, while 
        \[
            \sup_{|z|=1} \big\| T_{\widehat{m}_z}\colon L_p(\L \sG) \to L_p(\L \sG) \big\| = \infty.
        \] 
    \end{enumerate}
\end{ltheorem} 

\section{Radial Schur multipliers on trees} \label{Sect1}

\subsection{Trees and martingale transforms}
Let us briefly review a noncommutative martingale inequality which will be crucial in our argument. Given a von Neumann algebra $\M$, let $(\M_n)_{n \ge 0}$ be an increasing filtration---a sequence of von Neumann subalgebras such that $\bigcup_{n \ge 0} \M_n$ is w$^*$-dense in $\M$. Let $\mathbf{E}_n$ be the trace-preserving conditional expectation onto $\M_n$. Noncommutative $L_p$-martingales in this filtration are sequences $(x_n)_{n \ge 0}$ such that $x_n \in L_p(\M_n)$ and $\mathbf{E}_m(x_n) = x_m$ for $m \le n$. Set $\mathbf{d}_k(x) = \mathbf{E}_k(x) - \mathbf{E}_{k-1}(x)$ for $k \ge 1$ and $\mathbf{d}_0 = \mathbf{E}_0$ the martingale differences, so that we recover 
$x$ as a telescopic sum $\sum_k \mathbf{d}_k(x)$. In their pioneering paper \cite{PisierXu1997Martingales}, Pisier and Xu established a noncommutative extension of Burkholder-Gundy's inequality for the martingale square function \cite{BurkholderGundy1970}. By noncommutativity, this involves two forms of square functions---already in Lust-Piquard's work on noncommutative Khintchine inequalities \cite{Lust1986}---but we omit details here since they will not play any role below. A direct consequence is a noncommutative extension of Burkholder's inequality for martingale transforms 
\begin{equation} 
    \label{MartTransforms}
    \Big\| \sum_{k \ge 0} \alpha_k \mathbf{d}_k(x) \Big\|_{L_p(\M)} \lesssim \, p_\ast \, \Big( \sup_{k \ge 0} |\alpha_k| \Big) \, \Big\| \sum_{k \ge 0} \mathbf{d}_k(x) \Big\|_{L_p(\M)}  
\end{equation}
for any bounded sequence of scalars $(\alpha_k)_{k \ge 0}$. Pisier/Xu's original argument did not provide the optimal behavior of the constants shown above. Instead, this was found by Randrianantoanina in \cite{Randrianantoanina2002} with a remarkable proof of the noncommutative weak type (1,1) endpoint. We also refer to \cite{ParcetRandrianantoanina2006} for a simpler argument using a noncommutative form of Gundy's decomposition for martingale differences. It should be mentioned that inequality \eqref{MartTransforms} is still valid (using the same arguments) for martingales associated to a decreasing filtration.

\textit{The link with Schur multipliers. }
Given a $\sigma$-finite measure space $(\Omega,\mu)$, we shall write $S_p(\Omega)$ for the Schatten $p$-class of compact operators $A$ on $L_2(\Omega,\mu)$ with $\Tr |A|^p$ finite. When $(\Omega,\mu)$ is a discrete set with the counting measure and $M \colon \Omega \times \Omega \to \CC$ is bounded, the corresponding Schur multiplier acts on matrices by entrywise multiplication 
\[
    S_M(A) = \Big[ M(j,k) A_{j,k} \Big]_{j,k \in \Omega}.
\]
This is well-defined in the Hilbert-Schmidt class and we call it a Schur $S_p$-multiplier when it extends boundedly to the Schatten class $S_p(\Omega)$. Next, $S_M$ is completely $S_p$-bounded when additionally $S_M \otimes \mathrm{id}_{S_p(\Gamma)}$ extends to a bounded map on $S_p(\Omega \times \Gamma)$ for any countable index set $\Gamma$. Similar definitions apply when $\Omega$ is a $\sigma$-finite measure space using the density in $S_p(\Omega)$ of bounded linear maps on $L_2(\Omega,\mu)$ admitting a kernel, precise definitions can be found in \cite{LaffSall2011,Parcet2026ICM}.  Our measure space $(\Omega, \mu)$ in the first two sections of this paper will be a tree $X$ equipped with the graph metric $d$ and the counting measure. Let us fix a root in our tree $o \in X$ and set $|x| = d(x,o)$. Since $X$ is a tree, there exists a unique geodesic segment $[o,x]$ of length exactly $|x|$ for every $x \in X$. For each pair $(x,y)\in X\times X$, we will write $x \wedge y \in X$ for the last common ancestor of $x$ and $y$, that is, the unique vertex satisfying that the geodesic segments $[o,x]$ and $[o,y]$ intersect exactly in $[o,x\wedge y]$. 
For each $z\in \TT$, define the Schur multiplier
\[
    \C_z(A) = \Big[ z^{|x \wedge y|} A_{x,y} \Big]_{x,y \in X}.
\]

\begin{lemma} \label{KeyLemma}
If $1 < p < \infty$, we have
\[
    \sup_{z \in \TT} \big\| \, \mathcal{C}_z \colon S_p(X) \to S_p(X) \big\|_\cb \lesssim p_\ast.
\]
\end{lemma} 

\begin{proof} 
Given $k\geq 0$, consider the equivalence relation
\[
x\sim_k y \ \Leftrightarrow \ x=y \ \text{or } |x\wedge y| \geq k.
\]
Let $\mathcal{P}_k$ be the partition of $X$ induced by $\sim_k$ and define the map
\[
\EE_k(A) = \sum_{\Sigma \in \mathcal{P}_k} P_\Sigma A P_\Sigma = \Big[ \mathbf{1}_{x\sim_k y}\,A_{x,y} \Big]_{x,y},
\]
where $P_\Sigma$ is the orthogonal projection from $\ell_2(X)$ onto $\ell_2(\Sigma)$. The map $\EE_0$ is just the identity and each $\EE_k$ defines a conditional expectation onto decreasing matrix subalgebras of $\B(\ell_2(X))$. It is therefore clear that $(\EE_k(A))_{k \ge 0}$ defines a decreasing noncommutative martingale. Note as well that these conditional expectations approximate the diagonal restriction $\EE_\diag$---the only trace preserving conditional expectation onto the subalgebra of diagonal matrices---since the intersection of all  partitions $\mathcal{P}_k$ is the partition into singletons. If $\mathbf{d}_{k} = \mathbf{E}_k - \mathbf{E}_{k+1}$ for $k\geq 0$, then it is clear that we may write
\[
    A = \mathbf{E}_{\mathrm{diag}}(A) + \sum_{k\geq 0} \mathbf{d}_{k} (A).
\]
Next, if we set $\mathcal{D}_z = \mathrm{diag}(z^{|\cdot|})$ we claim that
\begin{equation} \label{Martingale transform}
   \mathcal{C}_z(A)= \mathcal{D}_z \cdot \mathbf{E}_{\mathrm{diag}}(A) + \sum_{k\geq 0} z^{k} \mathbf{d}_{k} (A).
\end{equation}   
Indeed, $\mathcal{C}_z$ acts by pointwise multiplication with $z^{|\cdot|}$ on the diagonal, whereas for $x\neq y$, we have $x\sim_k y \ \text{and } x\not\sim_{k+1} y \ \Leftrightarrow \ |x\wedge y| = k$. Thus, the off-diagonal entries of $\mathbf{d}_{k} (A)$ are exactly $\mathbf{1}_{|x\wedge y|=k}A_{x,y}$ while the diagonal ones vanish. Finally, the statement follows from the martingale transform inequality \eqref{MartTransforms} and the identity \eqref{Martingale transform}  
\[
    \| \mathcal{C}_z(A) \|_p 
    \leq \| A \|_p + \Big\| \sum_{k \geq 0} z^k \mathbf{d}_{k} (A) \Big\|_p 
    \lesssim \, \max\{ p, p' \} \, \Big( \sup_{k \geq 0} |z^k| \Big) \, \| A \|_p =  \, p_\ast \, \| A \|_p. 
\]
The same argument gives cb-bounds. This follows from the fact that martingale transforms are completely bounded, which follows automatically by taking the expectations $\EE_k \otimes \Id_{M_m}$ instead.
\end{proof}

\textit{Radial Schur multipliers on trees. }
Define 
\begin{equation}
    \label{eq:DeltaDef}
    \Delta_z(A) = \Big[ z^{d(x,y)} A_{x,y} \Big]_{x,y}.
\end{equation}
It is an elementary fact that
\begin{equation*}
    \label{eq:triangle}
    d(x,y) = |x| + |y| - 2 |x \wedge y|.
\end{equation*}
This implies that
\[
    \Delta_z = \mathcal{U}_z \circ \mathcal{C}_{z^{-2}}
\]
with $\mathcal{U}_z(A)_{x,y} = z^{|x|+|y|} A_{x,y}$. Since the latter is given by left and right multiplication by diagonal unitaries, we see that $\mathcal{U}_z$ are $S_p$-isometries, so that $\Delta_z$ and $\mathcal{C}_{z^{-2}}$ share the same norm and cb-norm on Schatten $p$-classes. We are now in position to prove the main result of this paper.

\begin{proof}[{Proof of Theorem \ref{Main Thm}}]
    Let $\Delta \colon S_p(X) \to L_p(\TT;S_p(X))$ be the map sending $A$ to the function $z \mapsto \Delta_z(A)$. Then, we will see that
    \[
        \frac1{c(p)} \, \| A \|_p \lesssim \| \Delta(A) \|_{L_p(\TT; S_p)} \lesssim c(p) \, \| A \|_p,
    \]
    with $c(p) = p_\ast$. Indeed, the fact that $\Delta$ is bounded follows from Lemma \ref{KeyLemma} and the identity $\Delta_z = \mathcal{U}_z \circ \mathcal{C}_{z^{-2}}$. Integrating on $z$ yields
    \[
        \| \Delta(A) \|_{L_p(\TT;S_p)} = \Big( \int_\TT \| \Delta_z(A) \|_{S_p}^p \, dz \Big)^{\frac1p} \lesssim c(p) \Big( \int_\TT \| A \|_{S_p}^p \, dz \Big)^{\frac1p} = c(p) \, \| A \|_{S_p}.
    \]
    To see the lower bound, just notice that $A = \Delta_{\bar z} \Delta_z(A)$ and integrate on the torus $\TT$. Now, given an even function $m: \ZZ \to \CC$ satisfying $M(x,y) = m(d(x,y))$ as in the statement, we have that the following diagram commutes
    \begin{equation}
        \label{eq:Intertwining}
        \begin{tikzcd}
            S_p(X) \arrow[r, "\Delta"] \arrow[d, "S_M"] & L_p(\TT;S_p(X)) \arrow[d, "T_m \otimes \Id"]\\
            S_p(X) \arrow[r, "\Delta"] & L_p(\TT;S_p(X))
        \end{tikzcd}
    \end{equation}  
In other words, we have $\Delta \circ S_M = (T_m \otimes \mathrm{id}) \circ \Delta$. Since $\Delta$ is a quasi-isometry, the boundedness of $T_m \otimes \Id$ implies the boundedness of $S_M$ with constant $c(p)^2$. The complete bound for $S_M$ follows from the fact that $\Delta$ and its inverse are completely bounded. To see the second assertion, let us first assume that $X$ contains an infinite ray starting at $o \in X$. Let us denote by $\rho = (\rho_n)_{n \geq 0}$ the sequence of nodes of this ray. The restriction of $M(x,y)$ to $\rho$ is just
    \[
        M(\rho_n, \rho_\ell) = m(d(\rho_n, \rho_\ell)) = m(|n-\ell|).
    \]
Restriction of Schur multipliers preserves the norm on Schatten $p$-classes \cite[Section 1]{LaffSall2011}. It turns out moreover that the restriction to $\rho \times \rho$ gives a Herz-Schur multiplier over $S_p(\ZZ)$. Thus, by Fourier-Schur transference \cite{NeuRic2011} over $\ZZ$ we obtain that: 
\[
    \big\| S_M \colon S_p(X) \to S_p(X) \big\| \geq \big\| S_{M_{\mid_{\rho \times \rho}}} \! \colon S_p(\NN_0) \to S_p(\NN_0) \big\| \geq \big\| T_m \colon L_p(\TT) \to L_p(\TT) \big\|,
\]
where we have used that, for even symbols $m$, the Herz-Schur multiplier norm is detected by just the positive entries on $S_p(\NN_0)$. This allows us to conclude. Observe as well that the inequalities above work similarly well for completely bounded norms. Now, if $X$ doesn't contain an infinite ray, because the diameter is infinite, it has to contain arbitrarily long geodesic segments starting at $o \in X$. By restricting to those segments and using that the $L_p$-norm of $T_m$ can be approximated when restricting the multiplier to $[0,N]$ as $N \to \infty$ we can conclude.
\end{proof} 

\begin{remark} \normalfont
The lower estimate holds as well in the category of Banach spaces---replacing cb by ordinary norms---whereas our argument for the upper bound uses critically the cb-boundedness of the Fourier multiplier $T_m \colon L_p(\TT) \to L_p(\TT)$. If the upper estimate also held for Banach space norms, $S_p$-bounded radial Schur multipliers on unbounded trees would coincide with the class of $L_p$-bounded Fourier multipliers on $\TT$. As pointed by Patrick Poissel, according to the existence from \cite[Chapter 8]{Pi1998} of Fourier $L_p$-multipliers on $\TT$ which fail complete boundedness, this would imply the existence of Schur $S_p$-multipliers which fail to be completely $S_p$-bounded, solving an open problem posed by Pisier in 1998. We refer to \cite{HuangSukochevTomskova2026Schur} for recent progress on the problem.    
\end{remark}

\begin{remark}\label{rmk:Optimalcp} \normalfont
    While the constant in \eqref{SphericalInequality} can easily be seen to be optimal, at present, wo do not know whether the $p$-dependent constant in Theorem \ref{Main Thm} can be improved. The fact that Theorem \ref{Main Thm} fails for $p=\infty$ implies that the constant must diverge when $p\to 1^+$ or $p \to \infty$. It is also evidently self-dual. Thus, it would be natural to conjecture that it can be lowered from $O(p_\ast^2)$ to $O(p_\ast)$, but we have only partial results in that direction. A naïve approach would consist on noticing that, if $\ev_0: C(\TT) \to \CC$ is the state given by evaluation at $0 \in \TT$, then formally it is satisfied that 
    \[
        (\ev_0 \otimes \Id) \circ (T_m \otimes \Id) \Delta = S_M, \, 
        \text{ where } M(x,y) = m(d(x,y)).
    \]
    However, $\ev_0$ is unbounded on $L_p(\TT)$. Similarly, one can show that, for every $A \in S_p(X)$, $z \mapsto \Delta_z(A)$ is norm continuous in $S_p(X)$. This follows from a careful application of Burkholder's inequality for martingale transforms \eqref{MartTransforms}. But, while this gives that $\Delta \colon S_p(X) \to C(\TT; S_p(X))$ is a bounded map, in order to apply the intertwining identity we will need for $T_m$ to be completely bounded on $C(\TT)$, leaving only non-singular multipliers obtained when $\widehat{m}$ is a finite signed measure in $M(\TT)$. A more sophisticated approach will be considered in a follow-up paper. 
\end{remark}

\subsection{General spherical functions on homogeneous trees}
Take $X$ to be the $(q+1)$-homogeneous tree, which we will denote by $X_q$, and let us again fix a root $o \in X$ on it. The spherical functions on $X_q$ are defined as radial functions that are formal eigenvectors of the adjacency operator 
\[
    \A_q f (x) = \frac1{q+1} \, \sum_{ y \, | \, x \sim y } f(y).
\]
More concretely, a function $\varphi: X_q \to \CC$ is \emph{spherical} iff
\begin{itemize}
    \item $\varphi(o) = 1$.
    \item There exists $z \in \CC$, such that $\A_q \varphi = z \, \varphi$.
    \item $\varphi$ is \emph{radial}, that is, there exists $\mathring{\varphi}: \NN_0 \to \CC$ such that $\varphi(x) = \mathring{\varphi}(d(x,o))$.   
\end{itemize}
We will refer to the parameter $z \in \CC$ as the eigenvalue of the spherical function and we will denote the functions involved by $\varphi_z$ and $\mathring{\varphi}_z$ whenever we want the relationship between the eigenvalue and its eigenfunction to be explicit. Recall that, if $x \in X_q$ is a vertex at distance $n$ from the root, then it is exactly adjacent to  $q$ vertices at distance $n+1$ and one at distance $n-1$. Thus, the definition of a spherical function can be translated into the following recurrence for its radial profile $\mathring{\varphi}_z$
\begin{equation}
    \label{eq:RecurrenceSpherical}
   q \, \mathring{\varphi}_z(n + 1) = (q + 1) \, z \, \mathring{\varphi}_z(n) - \mathring{\varphi}_z(n-1) \; \text{ for  } \; n \geq 1,
\end{equation}
while $\mathring\varphi(0) = 1$ and $\mathring\varphi(1) = z$. It is well known, see \cite[Theorem 0.3]{HaagSteenSw2009}, that the radial Schur multiplier on $S_\infty(X_q)$ given by
\begin{equation}
    \label{eq:qSphericalDef}
    \Delta_{q, z}(A) := \Big[  \mathring\varphi_z(d(x,y)) \, A_{x,y} \Big]_{x, y \in X}
\end{equation}
is $S_\infty$-bounded precisely when $z$ lives in the following elliptical region 
\[ 
    \bigg\{ z \in \CC : \Re(z)^2 + \Big( \frac{q+1}{q-1} \Big)^2 \, \Im(z)^2 < 1 \bigg\} \cup \{\pm 1\}.
\]
Let us denote $\DD_q \subset \CC$ the interior of that ellipse. In analogy to the discussion after \eqref{SphericalInequality} it would be natural to conjecture that $\Delta_{q, z}$ is  uniformly $S_p$-bounded when $z \in \overline{\DD}_q$. Here, we will present a direct elementary proof of that fact that gives the same constant as in \eqref{SphericalInequality}. We will need a few facts beforehand. First, the space of sequences satisfying the recurrence \eqref{eq:RecurrenceSpherical} is spanned by the exponentials $n \mapsto \alpha_+^n, \alpha_-^n$, where $\alpha_+$ and $\alpha_-$ are the roots, depending on $z$ and $q$, of the characteristic equation
\[
    q \, \alpha^2 - (q + 1) \, z  \, \alpha + 1 = 0.
\]
In the event that the characteristic equation has a single root with double multiplicity $\alpha_+ = \alpha_{-}$, then it is spanned by $\alpha^n$ and $n \alpha^n$. Elementary computations now yield that
\[
    \alpha_\pm = \frac{(q+1)z \pm \sqrt{(q+1)^2 z^2-4q}}{2q}.
\]
Thus, the two roots coincide when the discriminant $D = (q+1)^2 z^2-4q = 0$, that is when
\[
    z = \pm \frac{2 \sqrt{q}}{q+1}, \; \text{ so that } \; \alpha_+ = \alpha_- = \pm q^{- \frac12}.
\]
We will denote the values of $z$ causing double roots as $\pm c_q$.  Solving for the initial conditions of the recurrence gives
\begin{equation}
    \label{eq:ClosedFormulaSph}
    \mathring{\varphi}_z(n) 
    \, = \,
    \begin{cases}
        \displaystyle{\Big(1+\frac{q-1}{q+1}n\Big) \, \Big( \frac{\pm 1}{\sqrt{q}} \Big)^n}
            & \text{ when }  \displaystyle{z = \pm c_q}\\[10pt]
        \displaystyle{\Big( \frac12 + \frac{(q-1)}{2 \sqrt{D}}z \Big)\alpha_+^n + \Big( \frac12 - \frac{(q-1)}{2 \sqrt{D}}z \Big) \alpha_-^n}
            & \text{ otherwise }
    \end{cases}
\end{equation}
These explicit formulas are very well known in the spherical function literature, although they are normally presented with logarithmic parametrizations \cite{FigaTalamancaNebbia1991, HaagSteenSw2009}. Lastly, we will need the following elementary complex analysis lemma.

\begin{lemma}
    \label{lem:Derivatives}
    Given a tree $X$, the multiplier $\Delta_z \colon S_p(X) \to S_p(X)$, given by $\Delta_z(A)_{x, y} = z^{d(x,y)}$ for $|z| \leq 1$ is bounded for every $p \in (1,\infty)$ with the same bound as in \eqref{SphericalInequality}. Furthermore, the map $z \mapsto \Delta_z$ is holomorphic over the open disk and, given any $0 \leq r_0 < 1$, we have that
    \[
        \sup_{|z| \leq r_0} \Big\| \frac{\partial}{ \partial \, z} \Delta_z \colon S_p(X) \to S_p(X) \Big\|_\cb 
        \, \lesssim_{(r_0)} \, 
        \max \{ p, p'\} \quad \text{for every} \quad 1< p < \infty. 
    \]
\end{lemma}

\begin{proof}
    The proof is immediate. First, notice that if $|z| \leq 1$, then $z = \rho \, u$ with $|u|=1$. If $|\rho|=1$ there is nothing to prove. Meanwhile, if $\rho < 1$ we have that $\Delta_z$ is the composition of $\Delta_u$ and the radial Schur multiplier $\rho^{d(x,y)}$, which is positive definite and thus completely bounded for every $p \in [1,\infty]$.
    The second inequality follows from the fact that the symbol $z \mapsto z^{d(x,y)}$ is bounded and holomorphic when $z$ in the unit disk and that the derivative of an holomorphic map is also holomorphic. Observe as well that this is equivalent to the $S_p$-boundedness of the radial Schur symbols $d(x,y) \, z^{d(x,y)}$ for $|z| \leq r_0 < 1$.
\end{proof}

With those ingredients at hand, we can now proceed to the main result of this subsection.

\begin{theorem}
    \label{thmEllipsoid}
    The radial Schur multiplier $\Delta_{q,z}$ in \eqref{eq:qSphericalDef} satisfies that
    \[
        \sup_{z \in \overline{\DD}_q} \big\| \Delta_{q,z}\colon S_p(X_q) \to S_p(X_q) \big\|_\cb \, \lesssim \, p_\ast \quad \text{for every} \quad 1 < p < \infty.
    \]
\end{theorem}

\begin{proof}
    We divide the proof into two steps. First, we will show that $z \in \DD_q$ implies that both roots satisfy $|\alpha_\pm| \leq 1$ and then, we will see that \eqref{eq:ClosedFormulaSph} implies $S_p$-boundedness when $|\alpha_\pm|\leq 1$. For the first claim, notice that $\alpha_+ \alpha_- = \tfrac1{q}$. Thus, if one of the roots satisfies $|\alpha_+| \leq 1$ then the other satisfies $|\alpha_-|\geq \frac1{q}$, thus
    \begin{equation}
        \label{eq:Annulus}
        \Big\{ z : |\alpha_+| \leq 1 \text{ and } |\alpha_-| \leq 1 \Big\}
        \; = \; \Big\{ z : \tfrac1{q} \leq |\alpha_+| \leq 1 \Big\} =: \mathbf{A}_q.
    \end{equation}
    The formulas for $\alpha_\pm$ in terms of $z$ are obtained after solving $\alpha$ in terms of $z$ the characteristic equation. If we instead express $z$ in terms of $\alpha$, we obtain the so-called Joukowski map
    \[
        J_q(\alpha) = \frac1{q+1} \Big( q \alpha + \frac1{\alpha}\Big)
    \]
   of which $z \mapsto \alpha_\pm(z)$ are the two branches of its inverse. Thus, we can characterize the set in \eqref{eq:Annulus} as the image under $J_q$ the the annulus $\mathbf{A}_q$. Now, elementary computations show that $J_q$ maps $\mathbf{A}_q$ into $\overline{\DD}_q$. In fact every concentric circle is mapped into an ellipse up to the circle $\{|z|= \tfrac1{q}\}$, which is mapped to the interval $[-c_q, +c_q] \subseteq \DD_q$.
    
    Now, we just need to show that the formula \eqref{eq:ClosedFormulaSph} induces uniformly bounded $S_p$ multipliers when $|\alpha_\pm| \leq 1$. Let us first deal with the case of two different roots. Regrouping terms gives
    \[
        \mathring{\varphi}_z(n)
        =
        \frac12 (\alpha_+^n+\alpha_-^n)
        +b_z \, \frac{\alpha_+^n-\alpha_-^n}{\alpha_+-\alpha_-}
        \quad \text{where} \quad b_z=\frac{q-1}{2q}z.
    \]
    The first summand induces the radial Schur multiplier $\Delta_{\alpha_+} + \Delta_{\alpha_-}$, which is uniformly bounded on the disk by Lemma \ref{lem:Derivatives}. For the second term, first notice that the coefficient $b_z$ is bounded on $\overline{\DD}_q$ and thus irrelevant. Next, fix $\delta > 0$ to be determined later. We separate the cases $|\alpha_+ - \alpha_-| \geq \delta$ and $|\alpha_+ - \alpha_-| <\delta$. In the first case, we have that
    \[
        \Big\| \frac{\Delta_{\alpha_+}-\Delta_{\alpha_-}}{\alpha_+-\alpha_-} \Big\|_{\CB(S_p)} \leq \frac{2}{\delta} \sup_{z \in \overline{\DD}_q} \| \Delta_z \|_{\CB(S_p)}.
    \]
    For the second case, using that $\alpha_+ \, \alpha_- = \tfrac1{q}$, we get that if $\alpha_+$, $\alpha_-$ are close enough, then their modulus should be of the order of $\tfrac1{\sqrt{q}} < 1$ and thus, they can be taken to lay in a closed disk $|z|\leq r_\delta < 1$. By the mean value theorem we get
    \[
        \frac{\Delta_{\alpha_+}-\Delta_{\alpha_-}}{\alpha_+-\alpha_-} = \int_0^1 \frac{\partial \Delta_z}{ \partial \, z} \big( \alpha_- + t (\alpha_+ - \alpha_-) \big)\, dt.
    \]
    Taking operator norms and applying Lemma \ref{lem:Derivatives} we can conclude.
    For the double multiplicity case, it is enough to consider the radial multiplier of symbol
    \[
        M(x,y) = \Big(1+\frac{q-1}{q+1}\,d(x,y)\Big) \, q^{-\frac{d(x,y)}{2}}.
    \]
    Its $S_p$-boundedness follows from that of $d(x,y)\rho^{-d(x,y)}$, which is obtained from Lemma \ref{lem:Derivatives}. 
\end{proof}

\begin{remark}\normalfont
    Alternatively, we could have used Theorem \ref{Main Thm} as a black box and check that $n \mapsto \mathring{\varphi}_z(n)$ is an even multiplier on the torus. But this, in turn, is almost trivial, since it follows from \eqref{eq:ClosedFormulaSph} after expressing $\mathring{\varphi}_z(n)$ as a linear combination of two complex exponentials, which implies  $T_{\mathring{\varphi}_z(n)}$ acts as a combination of Poisson kernels. This, while straightforward, would give a suboptimal constant since we need to pay the extra $p_\ast$ in Theorem \ref{Main Thm}.
\end{remark}


\section{The failure of Fourier-Schur transference on free groups} \label{Sect2}


\subsection{Free balls and horoballs} 

Let $\FF_r$ be the free group on $r$ generators and equip it with its natural word length $| \cdot |$ induced by the generating set. Consider the tree $X$ given by the Cayley graph of $\FF_r$ and note that its graph distance satisfies $d(g,h) = |g^{-1}h|$. Observe that, while the Schur symbols of Herz Schur multipliers depend on $gh^{-1}$, so that restrict to Fourier multipliers on $\L \sG \subseteq \B(L_2 \sG)$, working with symbols depending on $g^{-1}h$ is inconsequential for the Schur norms since $g \mapsto g^{-1}$ is a measurable map preserving the measure class in every l.c. group \cite[Lemma 1.1]{Parcet2024Idempotentschur}. Thus, we will choose $g^{-1}h$ when right-invariant distances and objects are more standard in the broader literature. 

\begin{proof}[Proof of Theorem \ref{CorFree}\ref{itm:CorFree1}]
The free ball Schur multipliers satisfy $M_R(g,h) = \1_{[0,R]}(d(g,h))$ and Theorem~\ref{Main Thm} reduces the assertion---up to a constant $p_\ast^2$---to the cb-boundedness of the Fourier multiplier with symbol $\1_{[-R,R]}$ on $L_p(\TT)$. The latter can be expressed in terms of two Hilbert transforms, which yield an additional factor $p_\ast$ to the constant.
\end{proof}

Using that the radial subalgebra of the free group algebra $\L \FF_r$ is Abelian, first Kuhn and later Bo{\.{z}}ejko/Fendler (independently) \cite{Kuhn1984Convergence, BoFen2006} showed that the free ball multipliers $\1_{B_R}$ are not uniformly bounded on $L_p(\mathcal{L}\FF_r)$ when $\big| \tfrac12- \tfrac1{p} \big| > \tfrac1{6}$. Together with the theorem above and for $p$ in that range, this shows that the norm of Herz-Schur multipliers on $S_p(\FF_r)$ and of their associated Fourier multipliers are not comparable. Since the upper bound in \eqref{FSTransference} is known to hold for arbitrary l.c. groups, it follows from the closed graph theorem that there are (radial) symbols $m \colon \FF_r \to \CC$ failing Fourier-Schur transference \eqref{FSTransference} in that range of $p$'s. 

We now improve on the result above by constructing explicit symbols which additionally fail to be $L_p$-bounded as Fourier multipliers. The free horoballs defined in the Introduction witness the behavior of all free balls $B_R$ simultaneously. Contrary to the Euclidean case---where horoballs become half-spaces---$\HB_\rho$ and $\HB_\omega$ are eventually disjoint at infinity for $\rho \neq \omega \in \partial \FF_r$. Observe that the union defining $\HB_\rho$ is increasing. Thus, by $L_p$-density of trigonometric polynomials, the uniform boundedness of $\1_{B_R}$ as Fourier multipliers on $L_p(\L \FF_r)$ implies the boundedness of $m_\rho = \1_{\HB_\rho}$ by translation invariance. In fact, we have the following equivalence.

\begin{lemma} 
    \label{HoroballvsBall}
    Given $\rho \in \partial \FF_r$ and $p \neq 2$, the following are equivalent:
    \begin{enumerate}[label={\rm \textbf{(\roman*)}}, ref={(\roman*)}]
        \item \label{itm:HoroballvsBall.1} $\displaystyle \big\| T_{m_\rho} \colon L_p(\L \FF_r) \to L_p(\L \FF_r) \big\| < \infty$.
        \item \label{itm:HoroballvsBall.2} $\displaystyle \sup_{R > 0} \big\| T_{\1_{B_R}} \colon L_p(\L \FF_r) \to L_p(\L \FF_r) \big\| < \infty$. 
    \end{enumerate}
\end{lemma}

\begin{proof}
The implication \ref{itm:HoroballvsBall.2} $\Rightarrow$ \ref{itm:HoroballvsBall.1} has been already justified. For the converse, consider a branch $\rho = (\rho_n)_{n \ge 1} \in \partial \FF_r$ and let $x_n$ be the last letter of $\rho_n$. Fix a positive integer $R$ and take the partition $\FF_r = \mathrm{A}_R \cup \mathrm{B}_R \cup \{e\}$ with $\mathrm{A}_R$ the words starting with $x_{R+1}$ or $x_R^{-1}$ and $\mathrm{B}_R$ the complementary set in $\FF_r \setminus \{e\}$. Let $x_\perp$ be a generator or an inverse of a generator not belonging to $\{ x_R, x_{R+1}^{-1}, x_{R+2} \}$. Then 
\[
    T_{\1_{B_R}} \Big( \sum_{w \in \FF_2} a_w \lambda_w \Big) = a_e \lambda_e  + \lambda_{\rho_{R+1} x_\perp}^* T_{m_\rho} \Big( \sum_{w \in \mathrm{A}_R} a_w \lambda_{\rho_{R+1} x_{\perp} w} \Big) + \lambda_{\rho_R}^* T_{m_\rho} \Big( \sum_{w \in \mathrm{B}_R} a_w \lambda_{\rho_R w} \Big).
\]
Since $\sA_R$ and $\sB_R$ define free Hilbert transforms  \cite{MeiRicard2017FreeHilb}, we get 
\[
    \sup_{R>0} \| T_{\1_{B_R}} \|_{\B(L_p(\L \FF_r))} \leq 1 + 2 c_p \|T_{m_\rho} \|_{\B(L_p(\L \FF_r))}. \qedhere
\]
\end{proof}
   
\begin{remark} \normalfont
A more geometric argument is also possible at the cost of a quadratic bound $\|T_{\1_{B_R}}\|_{\B(L_p)} \le \|T_{m_\rho}\|_{\B(L_p)}^2$, rather than the linear bound provided above. Indeed, \emph{the free balls $B_R$ may be realized as the intersection of two translates of any given horoball $\HB_\rho$}---the traslation depending on the radius $R$. Indeed, fix a positive integer $R$ and let $x_n$ be the last letter of $\rho_n$ as above. Let $x_\perp$ be this time a generator or an inverse of a generator not belonging to $\{x_{R+1}, x_{2R}, x_{2R+1}^{-1}\}$. Define the translations $\mathrm{A}_R := \rho_R^{-1} \HB_\rho$ and $\mathrm{B}_R := x_\perp^R \rho_{2R}^{-1} \HB_\rho$. It is clear that the free ball $B_R$ is included in the intersection $\mathrm{A}_R \cap \mathrm{B}_R$. Now, given $\omega \in \FF_r \setminus B_R$, we claim that $w \notin \mathrm{A}_R$ when the first letter of $w$ is $x_\perp$ and $w \notin \mathrm{B}_R$ otherwise. This proves our assertion. 

To justify the first claim, assume that $\omega = x_\perp^k \omega'$ for some $k \ge 1$ and  $\omega \in \mathrm{A}_R \cap (\FF_r \setminus B_R)$. Since generic elements of $\mathrm{A}_R$ look like $\rho_R^{-1} \rho_n \eta$ for some $n \ge 1$ and some word $\eta$ with $|\eta| \le n$, they belong to $B_R$ for $n \le R$. When $n>R$, this would imply  $\eta = \rho_n^{-1} \rho_R \cdot \omega$ with the product $\cdot$ admitting no cancellations since $\omega = x_\perp^k \omega'$ and $x_\perp \neq x_{R+1}$. However, this would imply that $|\eta| \ge (n-R) + (R+1) > n$ contradicting the definition of $\mathrm{A}_R$.  Similarly, when $w$ does not start with $x_\perp$ we now show that $\omega$ cannot belong to $\mathrm{B}_R \cap (\FF_r \setminus B_R)$. Generic elements of $\mathrm{B}_R$ look like $x_\perp^R \rho_{2R}^{-1} \rho_n \eta$ for some $n \ge 1$ and some word $\eta$ with $|\eta| \le n$. If such $\omega$ belongs to $\mathrm{B}_R \cap (\FF_r \setminus B_R)$ we would find that the length of the words $\rho_n^{-1} \rho_{2R} \cdot x_\perp^{-R} \cdot w$ would be less than or equal to $n$ for every $n \ge 1$. Since $\omega$ does not start with $x_\perp$ and $x_\perp \notin \{x_{2R},  x_{2R+1}^{-1}\}$, the products $\cdot$ above admit no cancellations once more. This gives the lower bound $|n-R| + R + (R+1) > n$ for the length.     

\end{remark}
   
\begin{proof}[{Proof of Theorem \ref{CorFree}\ref{itm:CorFree2}}]
    Since $\HB_\rho$ is an increasing union of balls, the first assertion on the Schur multipliers $S_{M_\rho}$ follows from the uniform cb-boundedness of free balls in Theorem \ref{CorFree}\ref{itm:CorFree1} and the $S_p$-density of finite rank matrices. The second assertion follows from Lemma \ref{HoroballvsBall} and the failure of uniform $L_p$-boundedness of Fourier free balls for $|\frac12 - \frac1p| > \frac16$ from \cite{BoFen2006,Kuhn1984Convergence}.
\end{proof}

\begin{remark} \normalfont \label{Rem-Curvature} 
Let us fix a parameter $q>1$ and set $|w|_q$ to be the $q$-th analog of the word length. Namely, if $w = g_{j_1}^{k_1} g_{j_2}^{k_2} \cdots g_{j_n}^{k_n}$ is any word written in terms of the generators $g_1, g_2, \ldots$ with $j_1 \neq j_2 \neq \ldots \neq j_n$ and integer exponents, we define 
\[
    |w|_q = \Big( \sum_{\ell=1}^n |k_\ell|^q \Big)^\frac1{q}
\]
so that $|\cdot|_1$ recovers the word length. In \cite{MixedNorm2026} we use a rescaling argument to disprove the uniform $S_p$-boundedness of free $q$-balls for every $p \neq 2$ and $q>1$. It essentially shows that rescaled large free $q$-balls approximate sets admitting 2-dimensional sections which fail the zero-curvature condition in \cite{Parcet2024Idempotentschur} characterizing Schur idempotents. By the easy direction in Fourier-Schur transference---valid for nonamenable groups---free $q$-balls also fail uniform $L_p$-boundedness in the Fourier setting for $p \neq 2$. 
Before becoming aware of Theorem \ref{CorFree}\ref{itm:CorFree1}, one might have been tempted to treat the case \(q=1\) by adapting the preceding argument. Again by Fourier-Schur transference, this would imply the solution to the free ball multiplier problem left open in \cite{Kuhn1984Convergence, BoFen2006}. To our surprise, Theorem \ref{Main Thm} shows a drastic change of regime at the endpoint $q=1$ and free balls turn out to be curved in the Fourier setting---at least for $|\frac12 - \frac1p| > \frac16$--- but flat in the Schur sense of \cite{Parcet2024Idempotentschur}. It is a major challenge to find sufficient evidence of curvature---in the sense of  \cite{Fef1971}---to disprove the uniform \(L_p\)-boundedness of Fourier projections onto free balls for \(p\) close to \(2\).
\end{remark}

\subsection{Spherical functions in \texorpdfstring{$\FF_\infty$}{F-infinity}} \label{SphericalSect}

In this section we prove Theorem \ref{corFree2}, which provides new natural examples failing Fourier-Schur transference in the full range $1 < p \neq 2 < \infty$. We will prove a slightly stronger statement, namely that $T_{m_z}$ is unbounded when restricted to the (Abelian) radial part of $L_p(\L \FF_r)$. Recall that the convolution algebra of \emph{radial} functions on $\FF_r$ is Abelian. In fact, taking
\[
    \mu_n = \sum_{|\omega| = n} \delta_w \quad \mbox{and} \quad \chi_n= \sum_{|\omega| = n} \lambda_\omega,
\]
so that $\chi_n = \lambda(\mu_n)$, it is easy to check that $\mu_1 \ast \mu_1 = 2r + \mu_2$ and $\mu_1 \ast \mu_{n} = \mu_{n+1} + (2 r - 1) \mu_{n-1}$ for all $n \geq 2$. A routine diagonal argument gives that the radial algebra $\A_\rad \subset \L \FF_r$ is generated as a von Neumann algebra by $\chi_1$ and as such is Abelian. Since the work of Kesten \cite{Kesten1959} it is known that the spectrum of $\chi_1$ is $\sigma(\chi_1) = [-2 \sqrt{2 r - 1}, +2 \sqrt{2 r - 1} ]$. The restriction of the trace $\tau$ to the Abelian algebra $\A_\rad$ gives a measure over the interval that is explicitly known \cite{Kesten1959, FigaTalamancaPicardello1982}, we will refer to it as the \emph{radial Plancherel measure}. After scaling the spectrum of $\chi_1$ to the interval $[-1,1]$ we obtain  the Gelfand transform 
\[
    \Gamma: \big(\A_\rad, \tau_{\mid_{\A_\rad}}\big) \longrightarrow L_\infty\big([-1,+1], \nu_\tau \big),
\]
where the radial Plancherel measure is the probability measure 
\begin{equation}
    \label{eq:measure}
    d\nu_\tau(s)
    =\frac1{2\pi}\, \frac{4\tau\sqrt{1-s^2}}{1-4\tau(1-\tau)s^2}\,d s \qquad \mbox{with}  
    \qquad \tau = 1 - \frac1{2r} \qquad \mbox{and} \qquad -1\leq s\leq1.
\end{equation}
The $\Gamma$-image of the $L_2$ normalized coronas $\chi_n/\|\chi_n\|_2$ are 
\begin{equation*}
    \label{eq:initialQ}
    Q_0(s)=1,\qquad Q_1(s)=2\sqrt \tau\,s,\qquad
    Q_2(s)=4\sqrt \tau\,s^2-\frac1{\sqrt \tau},
\end{equation*}
and the family of orthogonal polynomials given by the recursion
\begin{equation}\label{eq:Qrecurrence} \nonumber 
 Q_{n+1}(s)=2sQ_n(s)-Q_{n-1}(s) \qquad \mbox{for} \qquad n\geq2.
\end{equation}
 
\begin{lemma} \label{prop:generating}
    For $|u|<1$ and $s\in[-1,1]$, we have
    \begin{equation}
        \label{eq:Gformula}
        \G_u(s) := \sum_{n=0}^{\infty}u^n Q_n(s)
         =
         \frac{1+2(\sqrt{\tau}-1)su+(1-\frac1{\sqrt{\tau}})u^2}{1-2su+u^2}.
    \end{equation}
    For $|u|=1$, \eqref{eq:Gformula} still makes sense by means of Abel summation.
Moreover, we get    
\begin{enumerate}[label={\rm \textbf{(\roman*)}}, ref={(\roman*)}]
        \item \label{itm:pNorm1} If $\alpha \in (0,\pi)$ and $\beta \in (0,1)$ then
        \begin{equation}\label{eq:interior-asymptotics}
            \| \G_{\beta e^{i \alpha}} \|_p
            \simeq_{(r,p,\alpha)}
            \begin{cases}
                \log\!\dfrac1{1-\beta},& \text{ when } p=1,\\[4mm]
                (1-\beta)^{\frac1{p}-1},& \text{ when } 1<p\leq\infty.
            \end{cases}
        \end{equation}

        \item \label{itm:pNorm2} Similarly, if $\alpha = 0$ and $\beta \in (0,1)$, it holds that
        \begin{equation}
            \label{eq:endpoint-asymptotics}
            \|\G_{\beta} \|_p
            \simeq_{(r,p)}
            \begin{cases}
                1, & \text{ when } 1\leq p<\dfrac{3}{2},\\[3mm]
                \Big( \log\!\dfrac1{1-\beta} \Big)^{\frac{2}{3}},& \text{ when } p=\dfrac32,\\[4mm]
                (1-\beta)^{\frac{3}{p}-2},& \text{ when } \dfrac32<p\leq \infty,
            \end{cases}
        \end{equation}
    \end{enumerate}
\end{lemma}

\begin{proof}
Define $\widetilde{Q}_0 := \frac1{\sqrt{\tau}}$ and $\widetilde{Q}_n := Q_n$ for $n \geq 1$. This gives 
    \begin{eqnarray*}
        \widetilde{\G}_u(s) \; := \; \sum_{n = 0}^\infty u^n \widetilde{Q}_n(s)
        & = & \frac1{\sqrt{\tau}} + 2 \sqrt{\tau} s u + \sum_{n = 2}^\infty \widetilde{Q}_n(s) u^n\\
        & = & \frac1{\sqrt{\tau}} + 2 \sqrt{\tau} s u + \sum_{n = 2}^\infty u^n \big[ 2 s \, \widetilde{Q}_{n-1}  (s)- \widetilde{Q}_{n-2}(s) \big] \\
        & = & \frac1{\sqrt{\tau}} + 2 \sqrt{\tau} s u + 2 s u \sum_{n = 1}^\infty \widetilde{Q}_{n} (s) u^n - u^2 \sum_{n=0}^\infty \widetilde{Q}_{n}(s) u^n\\
         & = & \frac1{\sqrt{\tau}} + 2 \sqrt{\tau} s u + 2 s u \, \Big[ \widetilde{\G}_u(s) - \frac1{\sqrt{\tau}} \Big] - u^2 \, \widetilde{\G}_u(s).
    \end{eqnarray*}
Identity \eqref{eq:Gformula} follows rearranging and using $$\G_u(s) = \widetilde{\G}_u(s) + \Big( 1 - \frac{1}{\sqrt{\tau}} \Big).$$
Next, we do the change of variable $s = \cos \theta$ to obtain 
    \begin{equation}\label{eq:theta-measure}
        d\nu_\tau(\cos\theta)
        =\frac{2\tau}{\pi}\,
        \frac{\sin^2 \theta}{1-4\tau(1-\tau)\cos^2\theta}\,d\theta.
    \end{equation}
    Since $\tau >\tfrac12$, we easily get $d\nu_\tau(\cos\theta) \simeq_{(r)} \sin^2 \theta \, d\theta$ since $4\tau(1-\tau) < 1$ and the denominator in \eqref{eq:theta-measure}  
is bounded below by a nonzero constant. Similarly, the numerator in \eqref{eq:Gformula} is comparable to a constant depending on $\tau$. Indeed, the function $F_{\tau}(s,u) = 1+2(\sqrt{\tau}-1)su+(1-\frac1{\sqrt{\tau}})u^2$ is clearly bounded above by 4 and observe that $$|F_\tau(s,u)-1| = \Big| (\sqrt{\tau}-1) u (2s+\frac{u}{\sqrt{\tau}}) \Big| \le (1-\sqrt{\tau}) \Big( 2 + \frac{1}{\sqrt{\tau}} \Big).$$ This implies $|F_\tau(s,u)| \ge 1-(1-\sqrt{\tau})(2+\frac{1}{\sqrt{\tau}}) = 2 \sqrt{\tau} - \frac{1}{\sqrt{\tau}} = C_\tau > 0$ and we get for $1\leq p<\infty$
$$\|\G_u \|_{p}^p \, \simeq_{(r,p)} \int_0^\pi \frac{\sin^2 \theta \, d\theta}{|1-ue^{i\theta}|^p|1-ue^{-i\theta}|^p} \qquad \mbox{for any} \qquad |u|\leq1.$$
Now we can prove the two final statements: 
    
    To prove {\ref{itm:pNorm1}}, fix $\alpha\in(0,\pi)$ and take $u=\beta e^{i\alpha}$.  Near $\theta=\alpha$
    \[
        \big| 1-\beta e^{i(\alpha-\theta)} \big|^2 \simeq (1-\beta)^2+(\theta-\alpha)^2,
    \]
    while the other factor in the denominator and $\sin\theta$ remain bounded above and below up to some constants depending on $\alpha$.  Therefore, since the integral above is smaller when $|\alpha - \theta|$ is large, we get the norm equivalence
    \[
        \|\G_{\beta e^{i\alpha}}\|_p^p
        \simeq_{(r,p,\alpha)}
        \int_{-1}^{1}\frac{d v}{((1-\beta)^2+v^2)^{p/2}}
    \]
    and we obtain the norm equivalence \eqref{eq:interior-asymptotics} by integrating the right hand side.

    For {\ref{itm:pNorm2}}, take $u=\beta \in(0,1)$.  Near $\theta=0$
    \[
        \big| 1-2\beta\cos\theta+\beta^2 \big|
        = \big| 1-\beta e^{i\theta} \big|^2
         \simeq (1-\beta)^2+\theta^2.
    \]
    Using the quadratic vanishing of $\sin^2 \theta$ near $\theta = 0$, one obtains
    \begin{equation}\label{eq:endpoint-integral}
        \|\mathcal G_{\beta} \|_p^p
         \simeq_{(r,p)}
        1+\int_0^1\frac{\theta^2}{((1-\beta)^2+\theta^2)^p}\,d\theta.
    \end{equation}
    Again, we obtain the norm equivalence \eqref{eq:endpoint-asymptotics} by integrating the right hand side.
\end{proof}

\begin{proof}[{Proof of Theorem \ref{corFree2}}]
    Observe that the Herz-Schur multiplier $H_{m_z}$ is an instance of the Schur multiplier $\Delta_z$ considered in Section \ref{Sect1}. In particular, the first assertion in Theorem \ref{corFree2} follows. Next, we show the unboundedness of $T_{m_z}$ by proving that its restriction to the radial subalgebra $\A_\rad$ is still unbounded. Using the Gelfand transform we may work instead with the operator on $L_2(\nu_\tau)$ given by $$\mathcal{J}_z = \Gamma \circ {T_{m_z}}_{\mid_{\A_\rad}} \circ \Gamma^{-1}$$ by radiality of the symbol $g \mapsto z^{|g|}$. Noting  that $$\mathcal{J}_z(\G_u) = \G_{z u},$$ the last assertion in the statement follows from the norm equivalences \eqref{eq:interior-asymptotics} and \eqref{eq:endpoint-asymptotics}.
\end{proof}

\begin{remark} \normalfont
For $u=e^{i\alpha}$ with $\alpha \in (0,\pi)$, the rational function $\G_u$ has a simple pole at $s=\cos\alpha$ where the measure has positive density.  Thus, it does not belong to $L_p(\nu_\tau)$ for any $p\geq1$.  At $u=\pm1$ the pole is at an endpoint and $\G_{\pm1} \in L_p(d\nu_\tau) \Leftrightarrow p<\frac32$. If $|u|>1$, the same rational formula defines a smooth function on $[-1,1]$; the denominator has no zero on that interval.  Approaching $\TT$ from the outside gives analogous asymptotics with $|\,|u|-1\,|$ in place of $1-\beta$.
\end{remark}

\section{Non-discrete algebraic groups and \texorpdfstring{$\sK$}{K}-biinvariant multipliers}

Assume that $\sG$ is a nondiscrete l.c.s.c. group admitting a lattice $\Gamma \leq \sG$, hence unimodular. Let $\Omega \subset \sG$ be a fundamental domain containing the identity $e$ in its interior. Pick $V \subset \Omega$ a symmetric open neighborhood of $e$ with $V V \subset \Omega$. It is possible to extend a function $m \colon \Gamma \to \CC$ to the whole group $\widetilde{m} \colon \sG \to \CC$ by 
\[
    \widetilde{m} = \frac{\1_V}{\sqrt{\mu(V)}} \ast \Big( \sum_{\gamma \in \Gamma} m(\gamma) \, \delta_\gamma\Big) \ast \frac{\1_V}{\sqrt{\mu(V)}},
\]
with $\mu$ the Haar measure on $\sG$. This extends $m$ and preserves the boundedness properties of $m$ by \cite[Lemma~2.6]{LaffSall2011}, both as a Fourier $L_p$-multiplier and as a Schur $S_p$-multiplier. Consider a symbol $m$ that induces an unbounded Fourier multiplier on $L_p(\L \Gamma)$ but a completely $S_p$-bounded Herz-Schur multiplier. According to \cite{LaffSall2011,WangXiaYao2025}, the extended symbol $\widetilde{m}$ leads as well to a completely $S_p$-bounded Herz-Schur multiplier on $S_p(\sG)$. On the other hand, the natural way to prove the unboundedness of the corresponding Fourier multiplier on $L_p(\L \sG)$ would be to use a de Leeuw subgroup restriction estimate in the spirit of \cite{CasParPerrRic2014}, but this requires that $\sG$ has the small almost $\Gamma$-invariant neighborhood property, which free group lattices like $\FF_2 \leq \SL_2(\RR)$ fail.  

\subsection{The group \texorpdfstring{$\mathbf{PGL}_\text{\textbf{2}}(\QQ_\text{\textbf{q}})$}{PGL(2,Qq)}}
Taking $\sG = \PGL_2(\QQ_q)$ and $\sK = \PGL_2(\ZZ_q)$, it is well known that $(\sG, \sK)$ forms a Gelfand pair. The space $X_q = \sG/\sK$ can be identified with a $(q+1)$-homogeneous tree, a particular case of a Bruhat-Tits building in a way that $\sK$-biinvariant functions on $\sG$, when lifted to $X_q$, correspond to radial functions. Furthermore, the Cartan length satisfies the identity $\ell_\mathrm{Car}(g) = d(g\sK, \sK)$, where $d: X_q \times X_q \to \NN_0$ is the graph distance. 

\begin{proof}[Proof of Theorem \ref{CorSimple}\ref{itm:CorSimple1}]
We have 
\[
  L_2(\sG)\simeq \ell_2(\sG/\sK) \otimes_2 L_2(\sK) = \ell_2(X_q) \otimes_2 L_2(\sK),
\]
so that $S_p(\sG) \simeq S_p(X_q) \otimes_p S_p(\sK)$. Since $\ell_\mathrm{Car}(g^{-1}h) = d(g^{-1}h\sK, \sK) = d(h\sK, g\sK)$, the complete boundedness of $S_{M_z}$ reduces to that of the Schur multiplier $\Delta_z$ with symbol $z^{d(x,y)}$, considered in Section \ref{Sect1}. For the unboundedness of the Fourier multipliers $T_{m_z}$, we follow the strategy in Subsection \ref{SphericalSect} using that the convolution algebra of $\sK$-biinvariant functions is Abelian. When taking weak-$\ast$ closures, this $\sK$-biinvariant algebra can be identified with the finite corner
\[
    \L(\sK \backslash \sG/\sK) \simeq p_\sK \, \L \sG \, p_\sK \qquad \mbox{ where } \qquad p_\sK = \int_{\sK} \lambda_k \, d \mu_\sK(k),
\]
and $\mu_\sK$ is the normalized Haar measure on $\sK$. The trace of $p_\sK$ is finite and thus we can identify such corner with its Gelfand spectrum and $\tau$ restricted to the $\sK$-biinvariant part with a finite measure
\begin{equation}
     \big( p_\sK \, \L \sG \, p_\sK, \tau \big) \longrightarrow L_\infty\big( [-2,2]; \mu_q\big),
\end{equation}
where 
\[
d\mu_q(x)=
 \frac{q+1}{(q^{1/2}+q^{-1/2})^2-x^2}
 \frac{\sqrt{4-x^2}}{2\pi}\,
\,dx \qquad \mbox{with} \quad -2 \le x \le 2.
\]
The explicit form of this isomorphism can be given by the spherical transform, which sends a suitable normalization of the characteristic functions of the spheres to the corresponding family of orthogonal polynomials. Namely, the $n$-th sphere
\[
    \sK \, \begin{bmatrix} q^n&0\\0&1\end{bmatrix} \, \sK
\]
is sent to $P_n(s)$, where we set 
\[
    P_0(s)=1,\qquad P_1(s)=s, \qquad P_n(s)=U_n(s/2)-q^{-1}U_{n-2}(s/2),
\]
where $U_n$ is the Chebyshev polynomial of the second kind and $n\ge2$, we refer to \cite{HasegawaSaigoSaitoSugiyama2018} and references therein for further details. Now, explicit computing of the generating function for $(P_n(s))_{n \geq 0}$ and its $L_p(\mu_q)$-norms for $\beta$ and $\beta e^{i \alpha}$ gives the result as in Lemma \ref{prop:generating}. 
\end{proof}

\subsection{Radial multipliers in \texorpdfstring{$\mathbf{PSL_{\text{\textbf{2}}}(\RR)}$}{PSL(2,R)}}

Given this negative result in the $q$-adic case, it is tempting to use a similar strategy over the real group $\PSL_2(\RR)$. In this section, we show that the analogous Schur multipliers are not uniformly $S_p$-bounded. Let 
\begin{equation} \label{L-length}
    L(g) = \mathrm{arccosh}(\tfrac{1}{2}\| g \|_{\mathrm{HS}}^2),
\end{equation}
where $\|\cdot\|_\mathrm{HS}$ denotes the Hilbert-Schmidt norm. This is a $\sK$-biinvariant length for the maximal compact subgroup $\sK=\mathrm{PSO}(2)$, which plays the role of the Cartan length above. Indeed, given a Cartan decomposition 
\[
    g = k_1 \, 
    \begin{bmatrix} 
        e^{s/2} &0\\
        0&e^{-s/2}
    \end{bmatrix} 
    \, k_2
\]
we have $L(g) = s$ as expected. The next result shows a different behavior of multipliers in $\PSL_2(\RR)$.  

\begin{theorem} 
    \label{ThmBallsSL2}
    Consider the symbols $m_s \colon \PSL_2(\RR) \to \CC$ given by $m_s(g) = e^{2 \pi i s L(g)}$ for any $s \in \RR$. 
    The family of Herz-Schur multipliers $\{H_{m_s} \colon s \in \RR\}$ is not uniformly bounded on $S_p$. In other words
    \begin{equation}
        \label{eq:NonUniformSL2}
        \sup_{s \in \RR} \big\| H_{m_s} \colon S_p(\PSL_2(\RR)) \to S_p(\PSL_2(\RR)) \big\| = \infty \quad \text{for every} \quad p \neq 2.
    \end{equation}
\end{theorem}

\begin{proof}
    Let us denote $\sG = \PSL_2(\RR)$ and $\sK = \PSO(2)$. Since $m_s\colon \sG \to \CC$ is $\sK$-biinvariant, $M_s(g,h) = m_s(g^{-1}h)$ naturally lifts to the homogeneous space $\sG/\sK$ as $\widetilde{M}_s$. Identifying $\sK \backslash \sG$ with the hyperbolic upper half space $\HH^2$, and using that $\sK \acts \sG/\sK$ acts by hyperbolic rotations around $i \in \HH^2$, we obtain that $\widetilde{M}_s$ is a radial multiplier on the Riemannian distance below
    \[
        \widetilde{M}_s(x,y) = e^{ 2 \pi i \, s \, d_{\HH^2}(x,y)}.
    \]
    We have the following limit, which holds uniformly over compact subsets of $\HH^2 \times \HH^2$:
    \[
        \lim_{\delta \to 0^+} \frac{d_{\HH^2}\big( \exp_i(\delta v), \, \exp_i(\delta w) \big)}{\delta} = |w - v|,
    \]
    where $v,w$ are tangent vectors at $i \in \HH^2$ and $\exp_i$ is the geodesic flow starting at $i$. Given $x \in \HH^2$ there is a unique geodesic connecting it with $i$. Let us denote by $\phi_\delta(x)$ the change of variable given by sending $x$ to the point $\phi_\delta(x)$ along the geodesic connecting $i$ and $x$ whose distance from $i$ is a fraction $\delta$ of $d_{\HH^2}(x,i)$. That is
    \[
        \phi_\delta(\exp_i(v)) := \exp_i(\delta v).
    \]
    The map $\phi_t$ is a smooth diffeomorphism and, in particular, it preserves the measure class. By \cite[Lemma 2.1.]{Parcet2024Idempotentschur} the operator $S_p$ norms of $\widetilde{M}_s(x,y)$ and $\widetilde{M}_s(\phi_t(x),\phi_t(y))$ coincide. Assuming by contradiction that the supremum on \eqref{eq:NonUniformSL2} is bounded, we can take $s = \tfrac{1}{\delta}$, so that
    \[
        \lim_{\delta \to 0^+} e^{ 2 \pi i \, \tfrac{1}{\delta} \, d_{\HH^2}(\phi_\delta(x),\phi_\delta(y))} 
        = e^{ 2 \pi i |v_x - v_y|},
    \]
    where $v_x$ represents the tangent vector of the unique geodesic from $i$ to $x$. By assumption, the $\delta$-dependent family of Schur multipliers above is uniformly bounded. The limit converges uniformly over compact sets. Thus, the limit is also an $S_p$-bounded Schur multiplier by  \cite[Theorem 1.19]{LaffSall2011}. But, since $\RR^2$ is an amenable group, by \cite{CasSall2015}, that would imply the boundedness of the half-wave propagator $e^{2 \pi i \, |\xi|}$, which is known to be unbounded as a Fourier multiplier on $L_p(\RR^2)$. We refer to \cite{Miyachi1980,Peral1980} or to \cite{BeltranRoosRutarSeeger2025} for a more modern treatment.
\end{proof}

\begin{remark} \normalfont
    \label{rmk:IntertwiningSp}
    There is an alternative proof of Theorem \ref{ThmBallsSL2}. Indeed, assume by contradiction that the family $\{H_{m_s}: s \in \RR \}$ is indeed uniformly bounded in $S_p(\sG)$. Then, it induces a bounded map $\Lambda \colon S_p(\sG) \to L_\infty(\RR; S_p(\sG))$ given by sending $A \in S_p(\sG)$ to the function $s \mapsto H_{m_s}(A)$. This map doesn't preserve the Lebesgue measure in $\RR$, since it is not a probability measure. Therefore it does not immediately extend to $L_p(\RR; S_p)$. However this can be fixed by choosing Gaussian probability measures $d \mu_\delta(x) = \gamma_\delta(x) dx$ of variance $\frac1{\delta}$. Then, by the finiteness of $\mu_\delta$, we have that $\Lambda_\delta: S_p(\sG) \to L_\infty(\RR, d\mu_\delta; S_p(\sG))$ are quasi-isometries. Now, we can use that multiplication by a Gaussian intertwines in $L_p$-sense with the action of the Fourier multiplier in the limit as $\delta \to 0$, ie:
    \[
        {\delta^{-\frac1{2p}}} \, \big\| \gamma_\delta^{\frac1{p}} (T_\varphi \otimes \Id)[F] -  (T_\varphi \otimes \Id)[\gamma_\delta^{\frac1{p}} F] \big\|_{L_p(\RR;S_p)} \to 0,
        \; \text{ as } \, \delta \to 0^+.
    \]
    This property goes back to de Leeuw's celebrated paper \cite{Lee1965}, but the exact same argument needed here can be found in \cite[Proposition 1.7]{ParcetSLn2018} or \cite[Section 2.5]{ParRog2016Hilbert}. Thus, in the limit $\delta \to 0$,  $\Lambda_\delta$ intertwines $H_{\varphi \circ L}$ with $T_{\varphi} \otimes \Id_{S_p}$ and thus
    \[
        \big\| H_{\varphi \circ L}\colon S_p(\sG) \to S_p(\sG) \big\|  \; \lesssim \; \big\| T_\varphi: L_p(\RR) \to L_p(\RR)\big \|_{\cb}.
    \]
    Taking $\varphi = \1_{[-R,R]}$, we would get the $L_p$-boundedness of the multiplier induced by geodesic balls on $\HH^2$, which contradicts the characterization of Schur idempotents with a smooth boundary in \cite{Parcet2024Idempotentschur}. 
\end{remark}

\subsection{\texorpdfstring{$\sK$}{K}-biinvariant counterexamples in \texorpdfstring{$\mathbf{PSL_{\text{\textbf{2}}}(\RR)}$}{PSL(2,R)}}

\newcommand{\Do}{\Omega}

Now we take $$(\sG, \Gamma, \sK) = \big( \PSL_2(\RR), \PSL_2(\ZZ), \PSO(2) \big).$$ 
Then $\Gamma = \big\langle a,b\mid a^2=b^3=e \big\rangle = \mathrm{A} \ast \mathrm{B}$, where $\mathrm{A}$ and $\mathrm{B}$ are respectively generated by \[
 a=\begin{bmatrix}0&-1\\1&0\end{bmatrix} \quad \text{and} \quad
 b=\begin{bmatrix}0&-1\\1&1\end{bmatrix}.
\]
Let \(T\) be its Bass--Serre tree, with vertex set \(\Gamma/\mathrm{A}\sqcup\Gamma/\mathrm{B}\) and edges $\{\gamma \mathrm{A},\gamma \mathrm{B}\} $ for $\gamma\in \Gamma$. Fix \(o=\mathrm{A}\) as root and denote by $d_{T}$ the standard distance on the tree $T$. We define the vertex-displacement length $\ell_T \colon \Gamma \to \NN_0$ by
\[
  \ell_T(\gamma)=d_{T}(o,\gamma^{-1}o)\quad \text{for} \quad \gamma\in\Gamma.
\]
$\ell_T$ is a quasi-length since it is null on $\mathrm{A}$. Nevertheless, it is symmetric and satisfies the triangle inequality $\ell_T(\eta \gamma) \leq \ell_T(\gamma) + \ell_T(\eta)$. Recall as well that $T$ can be interpreted geometrically as the orbit on the hyperbolic plane of the geodesic arc $[i,e^{\frac{2\pi i}{3}}]$ under the action of the modular group $\Gamma = \PSL_2(\ZZ)$. 
A reduced word $\gamma \in \Gamma$ is just an alternating succession of $a$ and $b^{\pm1}$. A simple computation gives that
\begin{equation} \label{Eq-nB}
    \ell_T(\gamma) = 2 n_\mathrm{B}(\gamma),
\end{equation}
where $n_\mathrm{B}(\gamma)$ counts the number of appearances of $b^{\pm1}$ when writing $\gamma$ in reduced form.

The quotient $\Gamma \backslash \sG$ admits a right $\sK$-invariant fundamental domain $\Do \subset \sG$. That is, a measurable set $\Do$ whose left $\Gamma$-translates $\{ \gamma \Do \}_{\gamma \in \Gamma}$ form a partition of $\sG$ up to a measure-zero set and where $\Do k = \Do$ for every $k \in \sK$.
Here, the set equality should be understood again up to null sets. The construction of such domain can be made explicit. $\sG$ has a unique $\sA\sN\sK$ decomposition, where the solvable group $\sA \sN$ can be identified as usual with the hyperbolic plane $\HH^2$. Take the standard fundamental domain of the isometric action of $\Gamma$ on $\HH^2$ given by
\[
    \mathcal{S} := \Big\{ z \in \HH^2 : -\tfrac12 \leq |\Re(z)| \leq \tfrac12, \, |z| \geq 1 \Big\} 
\]
and choose $\mathcal{S}_0 \subset \mathcal{S}$ by selecting the boundary points so that $\mathcal{S}_0$ is a Borel fundamental domain of the action of $\Gamma$. We can define $\Do := \mathcal{S}_0 \sK$ where $\mathcal{S}_0$ is seen as a subset living in the $\sA \sN$ part of the $\sA\sN\sK$ decomposition. 
We will also fix a Borel section $\sigma \colon \Gamma \backslash \sG \to \sG$ 
whose range agrees with \(\Do\) up to a null set. The associated right cocycle is the Borel map $\beta \colon \Gamma \backslash \sG \times \sG \to \Gamma$ determined by $\sigma(x)g=\beta(x,g)\sigma(x g)$. This map satisfies the cocycle identity
\[
    \beta(x,gh)=\beta(x,g)\beta(x g,h) \quad \text{and} \quad \beta(x,e) = 1 \quad \text{for} \quad (x,g,h) \in \Gamma \backslash \sG \times \sG \times \sG.
\]
Given $m \in \ell_\infty(\Gamma)$ its cocycle amplification $\widehat{m}: \sG \to \CC$ is 
\begin{equation}
    \label{def:CocycleInduction}
    \widehat{m}(g) = \int_{\Gamma \backslash \sG} m(\beta(x,g)) \, d\mu(x),
\end{equation}
where $\mu$ is the $\sG$-invariant measure on $\Gamma \backslash \sG$. We have the following elementary lemma (see e.g. \cite[\S 5]{WangXiaYao2025}). 
\begin{lemma}
    \label{lem:KbiinInduced}
    Given $m \in \ell_\infty(\Gamma)$, its amplification $\widehat{m}$ is continuous, bounded and $\sK$-biinvariant. 
\end{lemma}

\begin{proof}
    Let $\Do$ be the range of the section modulo null sets and take $d=\mu(\Do)$. Define the bounded measurable function $F_m(\gamma \omega)=m(\gamma)$ for $\gamma\in\Gamma$, $\omega \in \Do$. $F_m$ clearly satisfies $\| F_m \|_\infty \leq \| m \|_\infty$. Now
    \[
        \widehat m(g)=\frac1d\int_{\Do} F_m(\omega g)\,d\omega
        =\frac1d\int_\sG F_m(h)\1_{\Do}(hg^{-1})\,dh.
    \]
    This expresses $\widehat{m}$ as a convolution of elements in $L_\infty(\sG)$ and $L_1(\sG)$ and thus it is bounded and continuous. 
    For the $\sK$-biinvariance, we use that the geometric section above is the identity on the $\sK$-part of the $\mathrm{ANK}$ decomposition. Thus, the fact that $\sigma(x k) = \sigma(x) k$ for every $k \in \sK$ implies that
    \begin{align*}
        \beta(x,gk) =\beta(x,g) \quad \text{ and that } \quad
        \beta(x,kg) =\beta(xk,g).
    \end{align*}
    Hence, $\widehat m(gk)=\widehat m(g)$ follows immediately. For the right $\sK$-action, we have 
    \[
        \widehat m(kg)
        =\int_{\Gamma \backslash \sG} m(\beta(xk,g))\,d\mu(x)
        =\int_{\Gamma \backslash \sG} m(\beta(x,g))\, d\mu(x)
        =\widehat m(g).
        \qedhere
    \]
\end{proof}

Given $m_z(\gamma) = z^{\ell_T(\gamma)}$ as in the Introduction, the complete boundedness of $\widehat{m}_z$ as a Herz-Schur $S_p$-multiplier comes from the stability of Schur multipliers after induction \cite{LaffSall2011,WangXiaYao2025} and the uniform boundedness of the symbol $m_z$, which comes from Theorem \ref{Main Thm}. The negative part of Theorem \ref{CorSimple}\ref{itm:CorSimple2}---the fact that $\{T_{\widehat{m}_z} \colon z \in \TT\}$ is a non-uniformly bounded family of Fourier multipliers---is more involved. We will use the following weak limit law for the cocycle. Let $(a_t)_{t \in \RR}$ be the one-parameter group given by
\begin{equation}
    \label{eq:defAgroup}
    a_t = 
    \begin{bmatrix}
        e^{\frac{t}{2}} & 0\\
        0 & e^{-\frac{t}{2}}
    \end{bmatrix} 
    \in \sG.
\end{equation}
\begin{proposition}
    \label{prp:cocycleWeakLaw}
    The following holds in $\mu$-probability 
    \begin{equation} \nonumber
        \lim_{t \to \infty} \frac{\ell_T \big( \beta( \, \cdot \, , a_t ) \big)}{t \log t} = \frac{12}{\pi^2} 
    \end{equation}
    for any cocycle $\beta: \Omega \times \sG \to \Gamma$ associated to a measurable cross section of the quotient $\Gamma \backslash \sG$. As a consequence, the following limit holds for every $s \in \RR$ and any sequence $(R_j)_{j \ge 1}$ with $R_j \to \infty$ as $j \to \infty$
    \begin{equation} \nonumber
        \lim_{j \to \infty} \widehat{m}_{z_j}(a_{R_j s}) = e^{i s}
        \quad \text{where} \quad z_j = e^{i \tfrac{\pi^2}{12 \, R_j \log R_j}}. 
    \end{equation}
\end{proposition}

We will postpone its proof until the end of this section and first prove Theorem \ref{CorSimple}\ref{itm:CorSimple2}. Let us fix some notation. First, by normalizing the Haar measure on $\sG$, we can ensure that any $\sK$-biinvariant function $f \in L_1(\sG)$ satisfies the following
\[
    \int_\sG f(g) \, dg
    \, = \,
    \int_0^\infty f(a_t) \, \sinh(t) \, dt.
\]
This establishes an isomorphism between $L_p(\sK\backslash \sG / \sK)$ and $L_p(\RR_+, \sinh(t) \, dt)$ for every $p \in [1,\infty]$ whose inverse is given by $h \mapsto 
h \circ L$ with $L$ from \eqref{L-length}. As above, $\sK$-biinvariant functions span a commutative algebra under convolution, whose weak-$\ast$ closure is 
\[
    p_\sK \, \L \sG \, p_\sK \quad \text{with} \quad p_\sK = \int_\sK \lambda(k) \, d \mu_K (k).
\]
Here $\mu_\sK$ is the Haar measure on $\sK$. The Gel'fand transform of $\A = p_\sK \, \L \sG \, p_\sK$ gives
\[
    \big( p_\sK \, \L \sG \, p_\sK, \tau{|}_\A \big) 
    \xrightarrow{\;\; \mathrm{Sph} \;\;} 
    L_\infty(\RR_+, d \pi),
\]
where the pushforward of the restricted trace is the spherical Plancherel measure $d \pi$ given by $d\pi(\lambda)=\lambda\tanh(\pi\lambda)\,d\lambda$. Since $p_\sK L_2(\L \sG) p_\sK \ni \lambda(f) \mapsto f \in L_2(\sK \backslash \sG / \sK)$ is a unitary $L_2$-isometry by Plancherel, the spherical transform $\mathrm{Sph}$ extends to an $L_2$-isometry 
\[
    \Sph \colon L_2(\sK\backslash\sG/\sK)
    \longrightarrow 
    L_2(\RR_+,d\pi)
\]
with the following explicit formula
\begin{equation} 
    \label{eq:Mehler}
    (\Sph f)(\lambda)
    =
    \int_0^\infty f(a_t) \, \varphi_\lambda(t) \, \sinh t\,dt
    \qquad \text{where} \qquad
    \varphi_\lambda(t)
    =
    \frac{\sqrt2}{\pi}
    \int_0^t\frac{\cos(\lambda u)}{\sqrt{\cosh t-\cosh u}}\,du.
\end{equation}
We refer to \cite{Helgason2000} for the general theory of the spherical transform and to \cite[Chapter~V]{Lang1985SL2R} for the spherical Plancherel measure. The integral formula for $\varphi_\lambda$ that we have recalled above can be checked in \cite[eq. (10)]{Anker2017Dunkl}. Since $L_p(\A, \tau{|}_\A) \simeq L_p(\RR_+, d\pi)$ and any $\sK$-biinvariant Fourier multiplier $T_m$ preserves $L_p(\A)$, it holds that
\begin{equation} \label{Eq-Spectral}   \big\| \T_m: L_p(d\pi) \to L_p(d\pi) \big\| 
    \leq 
    \big\| T_m: L_p(\L \sG) \to L_p(\L \sG) \big\|,
\end{equation}
where $\T_m(f)$ is the spectral multiplier operator determined by $\T_m h = \Sph [ m \cdot (\Sph^{-1} h) ]$. 

\begin{lemma}
    \label{lem:spherical-limit}
    There is a constant $C > 0$ such that
    \begin{enumerate}[label={\rm \textbf{(\roman*)}}, ref={\rm (\roman*)}]
        \item \label{pslphase:eq:phi-bound}
        $|\varphi_\lambda(t)|\leq\varphi_0(t) \leq C(1+t)e^{-t/2}$ for every $\lambda, t\geq0$. 
        \item \label{pslphase:eq:phi-limit} 
        $\displaystyle{\lim_{R\to\infty} \tfrac{1}{R} e^{\frac{R s}{2}} \varphi_{\frac{\xi}{R}}(Rs) = \frac{2\sin(\xi s)}{\pi\xi}}$ for every $s>0$ and $\xi\geq0$.
    \end{enumerate}
\end{lemma}

\begin{proof}
    Point \ref{pslphase:eq:phi-bound} follows from \eqref{eq:Mehler}. Indeed, using that
	\[
	   \cosh t-\cosh u
	   =
	   2\sinh\left(\frac{t+u}{2}\right)
	   \sinh\left(\frac{t-u}{2}\right)
	   \geq
	   \frac{t^2-u^2}{2},
	\]
	we see that $\varphi_0(t) \leq 1$ for $0 < t \le 1$. Assume now that \(t\geq1\) and set \(v=t-u\). Note that
	$\sinh(t-v/2)\geq c e^{t-v/2}$ for an absolute constant $c>0$. Consequently, we deduce the inequalities below
	\begin{equation}\label{eq:dom} 	 	 \frac1{\sqrt{\cosh t-\cosh(t-v)}}
	 	 \lesssim e^{-t/2}
	 	 \frac{e^{v/4}}{\sqrt{\sinh(v/2)}}\lesssim e^{-t/2} (1+v^{-1/2})
	\end{equation}	
    and obtain
	\[
	\varphi_0(t)
	\lesssim e^{-t/2}\int_0^t(1+v^{-1/2})\,dv
	\lesssim (1+t)e^{-t/2} \quad \text{for all} \quad t \ge 0.
	\]
    The limit \ref{pslphase:eq:phi-limit} now follows easily from \eqref{eq:dom} and the dominated convergence theorem.
\end{proof}

Let $\{h_j\colon \sG \to \CC \mid j \ge 1 \}$ be a family of $\sK$-biinvariant functions and $R_j > 0$ a sequence of scalars with $R_j \to \infty$. Assume that  $h_j(a_{R_j s}) \to \Phi(s)$ for some $\Phi: \RR \to \CC$. We would like to translate uniform $L_p$-boundedness for the family of Fourier multipliers associated to $\{h_j\}$ into some limiting spectral multiplier result for $\Phi$. To that end, we introduce the following dilated versions of the measure $d\pi$
\begin{equation} \nonumber
    \label{pslphase:eq:rescaled-measures}
    d\pi_j(\lambda)=w_j(\lambda)\,d\lambda
    \quad
    \text{ where } 
    \quad
    w_j(\lambda)=R_j\lambda\tanh\Big(\frac{\pi\lambda}{R_j}\Big),
\end{equation}
while the limit measure would be set to $d\pi_\infty(\lambda)=\pi\lambda^2\,d\lambda$. Since $\tanh v\leq v$ for $v\geq0$, we deduce that $0 \leq w_j(\xi) \leq \pi\xi^2$ and  $w_j(\xi) \to \pi\xi^2$. For $1\leq p<\infty$, define $J_{j,p}:L_p(d\pi_j)\to L_p(d\pi)$ by the identity
\begin{equation} \nonumber
    \label{pslphase:eq:J}
J_{j,p}F(\lambda)=R_j^{\frac{3}{p}}F(R_j\lambda).
\end{equation}
This is a surjective isometry with inverse $J_{j,p}^{-1}f(\lambda) = R_j^{-\frac{3}{2}} f(\lambda \, R_j^{-1})$. In particular, it turns out that the dilated operators $\Lambda_j = J_{j,p}^{-1} \T_{h_j} J_{j,p}$ are independent of $p$ and they satisfy that the following isometry
\begin{equation} \label{Eq-Spectral2}    \big\| \Lambda_j \colon L_p(d\pi_j)\to L_p(d\pi_j) \big\|
    = 
    \big\| \T_{h_j}: L_p(d\pi)\to L_p(d\pi) \big\|.
\end{equation}
Next, let $U_j \colon L_2(d\pi_j) \to L_2(\RR_+, ds)$ for $j \in \NN \cup \{\infty\}$ be the isometries 
\begin{equation} \nonumber
    \label{pslphase:eq:Uj}
    U_jF(s)=\sqrt{R_j\sinh(R_js)}\,
            (\Sph^{-1}J_{j,2}F)(a_{R_js})
\end{equation}
for finite $j$ and 
\begin{equation}
    \label{eq:Uinf}
    U_\infty F(s)=\sqrt2\int_0^\infty F(\xi)\xi\sin(\xi s)\,d\xi.
\end{equation}
To see that the maps $U_j$ are unitary it is enough to notice that they are the composition of the unitaries $J_{j,2}$ and $\Sph^{-1}$ and of the spherical dilation $f(a_s) \mapsto f(a_{R_j s})$ with the precise normalizing factor that makes it isometric in $L_2$. Observe also that, up to constants $U_\infty$ is the odd part of the Fourier transform in $\RR$ composed with the unitary map from $L_2(d\pi_\infty)$ to $L_2(d\xi)$ sending $F(\xi) \mapsto \sqrt\pi\,\xi F(\xi)$. We will need the following.

\begin{lemma}
    \label{lem:U-convergence}
    Let $\{ h_j \colon \sG \to \CC \mid j \ge 1 \}$ be a uniformly bounded sequence of $\sK$-biinvariant functions. Assume there are scalars $R_j \to \infty$ such that $\lim_{j} h_j(a_{R_j s}) = \Phi(s)$ for some funcion $\Phi \in L_\infty(\RR_+)$. Taking $M_\Phi(f) = \Phi f$, let $B_\Phi := U^{-1}_\infty M_{\Phi} U_\infty$ be the spectral multiplier associated to the unitary $U_\infty$. Then, the following properties hold: 
    \begin{enumerate}[label={\rm \textbf{(\roman*)}}, ref={\rm {(\roman*)}}]
        \item \label{itm:U-convergence1} For every $F \in C_c^\infty(\RR_+)$, $U_jF\longrightarrow U_\infty F$ in $L_2(ds)$-norm.
        \item \label{itm:U-convergence2} For every $F, G \in C_c^\infty(\RR_+)$, we have $\displaystyle{\langle \Lambda_j F, G\rangle \to \langle B_\Phi F, G\rangle}$.
    \end{enumerate}
\end{lemma}

\begin{proof}
    For \ref{itm:U-convergence1}, use spherical inversion 
    \[
\Sph^{-1}h(a_t) = \int_0^\infty h(\lambda) \varphi_\lambda(t) \, d\pi(\lambda)
    \]
    to get for  $\xi=R_j\lambda$ that
    \begin{align*}
        (U_jF)(s)
        &=\sqrt{e^{-R_js}\sinh(R_js)}
        \int_0^\infty F(\xi)w_j(\xi)
            \left(\frac{e^{R_js/2}}{R_j}
            \varphi_{\xi/R_j}(R_js)\right)\,d\xi.
    \end{align*}
    For fixed $s>0$ and according to Lemma~\ref{lem:spherical-limit}\ref{pslphase:eq:phi-bound}, we make take limits under the integral. The factor outside the integral tends to $1/\sqrt2$ while the weight $w_j(\xi)$ converges to $\pi\xi^2$. Finally, Lemma~\ref{lem:spherical-limit}\ref{pslphase:eq:phi-limit} gives the limit of the factor in parenthesis. This justifies pointwise convergence. Moreover, weak $L_2$-convergence easily follows since $\{U_jF \colon j \ge 1\}$ is uniformly $L_2$-bounded, Now, the assertion follows by noticing that $\|U_jF\|_2 = \|F\|_{L_2(\RR_+,d\pi_j)} \to \|F\|_{L_2(\RR_+,d\pi_\infty)} = \|U_\infty F\|_2$ and $$\big\| U_\infty F - U_j F \big\|_{L_2(\RR_+)}^2 = \big\| U_\infty F \big\|_{L_2(\RR_+)}^2 + \big\| U_j F \big\|_{L_2(\RR_+)}^2 - 2 \, \mathbf{Re} \, \big\langle U_\infty F ,U_j F \big\rangle \longrightarrow 0.$$ 

    For \ref{itm:U-convergence2}, conjugating $\T_{h_j}$ by the unitary $U_j$ gives the key identity
    \begin{equation} \nonumber
        \label{pslphase:eq:Aj-unitary}
        \Lambda_j=U_j^{-1}M_{h_j(a_{R_j\cdot})}U_j.
    \end{equation}
    Therefore
    \[
        \langle \Lambda_jF,G\rangle_{L_2(d\pi_j)}
        =\int_0^\infty h_j(a_{R_js})(U_jF)(s)
                        \overline{(U_jG)(s)}\,ds.
    \]
    By \ref{itm:U-convergence1}, we obtain that $$\big\| U_\infty F\overline{U_\infty G} - U_jF\overline{U_jG} \big\|_{L_1(\RR_+)} \to 0.$$ By the uniform boundedness of $|h_j|$, we can replace the product in the pairing by its limit.
\end{proof}
 
\begin{proof}[Proof of Theorem \ref{CorSimple}\ref{itm:CorSimple2}]
    Let $m_z \colon \Gamma \to \CC$ be the map $m_z(\gamma) = z^{d_T(o,\gamma \cdot o)}$. By Theorem \ref{Main Thm} we have that $M_z(x,y) = {z^{d_T(x,y)}}$ defines a family of uniformly completely bounded Schur multipliers on $S_p(T)$ for every $1 < p <\infty$. Denote by $\phi \colon \Gamma \to T$ the map $\gamma \mapsto \gamma \cdot o$ from $\Gamma$ into its Bass-Serre tree. Since $$m_z(\gamma^{-1} \eta) = z^{d_T(\eta \cdot o, \gamma \cdot o)} = M(\phi(\gamma), \phi(\eta))$$ and we know that the $S_p$-norm of the Schur multipliers with symbols $M$ and $M \circ (\phi \times \phi)$ is the same \cite[Section 1]{LaffSall2011}, we deduce that $\{m_z \colon |z|=1 \}$ defines a family of uniformly completely bounded Herz-Schur multipliers on $S_p(\Gamma)$. Its cocycle amplification $\widehat{m}_z$ is a continuous $\sK$-biinvariant function by Lemma \ref{lem:KbiinInduced} and, by \cite{LaffSall2011,WangXiaYao2025}, the uniform complete boundedness of $H_{\widehat{m}_z}$ on $S_p(\sG)$ is preserved.

    For the Fourier side, assume by contradiction that 
     \[
            \sup_{|z|=1} \big\| T_{\widehat{m}_z}\colon L_p(\L \sG) \to L_p(\L \sG) \big\| < \infty.
        \]
        Choose a sequence $R_j \to \infty$, phases $z_j = \exp(i \pi^2 / 12 R_j \log R_j)$ and fix $h_j(g) = \widehat{m}_{z_j}(g)$. According to Proposition~\ref{prp:cocycleWeakLaw}, we are in the hypotheses of Lemma \ref{lem:U-convergence}. But the uniform boundedness of the multipliers $T_{\widehat{m}_z}$ on the Abelian corner $p_\sK \, L_p(\L \sG) \, p_\sK$, together with \eqref{Eq-Spectral} and \eqref{Eq-Spectral2} plus Lemma \ref{lem:U-convergence}\ref{itm:U-convergence2} imply the boundedness of the operator $B_\Phi = U_\infty^{-1}M_{\Phi} U_\infty$ for $\Phi(s) = e^{is}$ on $L_p(d\pi_\infty)$. To obtain a contradiction, we will show that $B_\Phi$ is indeed unbounded. First, observe from \eqref{eq:Uinf} that $U_\infty = \Fou_{\sin} \circ \V$, where $\V \colon L_2(d \pi_\infty) \to L_2(d\xi)$ is the operator $(\V f)(\xi) \mapsto \sqrt{\pi} \xi f(\xi)$ and $\Fou_{\sin}$ is the odd part of the Fourier transform on $\RR$
    \[
        (\Fou_{\sin} g)(s)=\sqrt{\frac2\pi}
            \int_0^\infty g(\xi)\sin(\xi s)\,d\xi.
    \]
    Let $g\in C_c^\infty(\RR_+)$ be real-valued with support away from $0$ and set $F(x)=g(x)/x$. It holds that $x B_\Phi F(x)=\Fou_{\sin} \big( \Phi \Fou_{\sin} (g) \big)(x)$. Using Fourier inversion and taking real parts gives the following identity when $\Phi(s) = e^{is}$
    \begin{equation}
        \label{eq:translation-identity}
        \Re\big(x B_\Phi F(x)\big)
        =\frac12\big[g_{\mathrm{odd}}(x+1)+g_{\mathrm{odd}}(x-1)\big] \quad \text{for} \quad x>0.
    \end{equation}
    Here $g_{\mathrm{odd}}$ is the odd extension of $g$. Now, fix a nonzero real-valued function
	\(\psi\in C_c^\infty(\RR_+)\) supported in $(1,2)$ and let \(0<\delta<1/4\).	
	Assume first that \(1\leq p<2\) and consider the functions 
	$
	g_\delta(x)=\psi(x/\delta)$ and $
	F_\delta(x)= g_\delta(x)/x$. Then we claim that  
    \[
        \frac{\|B_\Phi F_\delta\|_{L_p(d\pi_\infty)}}{\|F_\delta\|_{L_p(d\pi_\infty)}} \to \infty,
    \]
	which implies the assertion. To prove this claim note that
	\begin{equation*} 
		\|F_\delta\|_{L_p(d\pi_\infty)}^p
		=
		\pi \delta^{3-p}
		\int_1^2|\psi(u)|^p u^{2-p}\,du.
	\end{equation*}	When \(1+\delta<x<1+2\delta\), one has $
	g_\delta(x+1)=0$ and $
	g_\delta(x-1)
	=
	\psi\left(\frac{x-1}{\delta}\right)$. Thus \eqref{eq:translation-identity} gives
\begin{eqnarray*}		
\big\| B_\Phi F_\delta \big\|_{L_p(d\pi^2)}^p & \ge & \ \pi \int_{1+\delta}^{1+2\delta} \big| x B_\psi F_\delta(x) \big|^p \, x^{2-p} dx \\ 
		& \ge & \frac{\pi}{ 2^p}
		\int_{1+\delta}^{1+2\delta}
		\left|
		\psi\left(\frac{x-1}{\delta}\right)
		\right|^p
		x^{2-p}\,dx \ \ge \
		c_{p,\psi}\delta.
\end{eqnarray*}
	Thus
	\[
	\frac{\|B_\Phi F_\delta\|_{L_p(d\pi_\infty)}}
	{\|F_\delta\|_{L_p(d\pi_\infty)}}
	\geq
	c_{p,\psi}\delta^{(p-2)/p}
	\longrightarrow\infty,
	\]
	as claimed. The case \(p>2\) follows similarly for the same function $F_\delta(x) = g_\delta(x)/x$ but with the new choice \(g_\delta(x)=\psi((x-1)/\delta)\) and the interval \(\delta<x<2\delta\). In this case, we get the asymptotics reversed $\|F_\delta\|_p^p \approx \delta$ and $\|B_\Phi F_\delta\|_p^p \approx \delta^{3-p}$, which gives the right growth $\delta^{(2-p)/p}$ for $p>2$.  
\end{proof}

\begin{remark} \normalfont
    As stated, Theorem \ref{CorSimple}\ref{itm:CorSimple2} just shows that there are  families of uniformly (completely) $S_p$-bounded Herz-Schur multipliers---with $\sK$-biinvariant symbols---that are non-uniformly bounded as Fourier multipliers over $\PSL_2(\RR)$. However, given $p \in (1,\infty)\setminus \{2\}$, this easily implies that there are $\sK$-biinvariant completely $S_p$-bounded Herz-Schur multipliers that are not bounded as Fourier multipliers on $L_p(\L \PSL_2 \RR)$. Indeed, for every $p \in [1,\infty]$ the Fourier multiplier norm bounds the Herz-Schur one---see the discussion around block B in the Introduction---and thus the identity map $m \mapsto m$ induces a contractive inclusion from Fourier multipliers into Herz-Schur ones. The graph of this contractive inclusion is closed and, by the closed graph theorem, the fact that we cannot bound the inverse map implies that the inclusion is not surjective.
\end{remark}

\begin{remark} \normalfont
    \label{rem:discrete-modular-phases}
    In fact, given $z\in \TT$ and $p \in (1,\infty) \setminus \{2\}$ 
    \[
        \big\| T_{m_z}\colon L_p(\L \Gamma) \to L_p(\L \Gamma) \big\| \; < \; \infty
        \qquad \Longleftrightarrow \qquad
        z^4 = 1.
    \]
    In the exceptional cases $z^4 = 1$, $T_{m_z} $ is actually completely bounded. The above fact refers to the lattice $\Gamma$, not the full Lie group $\sG$. This is not used in the proof of uniform unboundedness of $T_{\widehat{m}_z}$, nevertheless we include it for completeness. For the negative assertion, we restrict to the index-three subgroup $H= \ker q$, where $q:\Gamma = \mathrm{A} \ast \mathrm{B} \to \mathrm{B}$ sends $a$ to $e$ and fixes $\mathrm{B}$. Writing $c_j=b^jab^{-j}$ 
    \[
        H= \big\langle c_0,c_1,c_2\mid c_0^2=c_1^2=c_2^2=e \big\rangle \; \simeq \; \ZZ_2 \ast \ZZ_2 \ast \ZZ_2. 
    \]
    A reduced nonempty word is of the form $h=c_{i_1}\cdots c_{i_n}$ with $i_j \neq i_{j+1}$ for $j \in \{1,2, \dots, n-1\}$ and thus corresponds to a word $b^{i_1} \, a \, b^{i_2-i_1} \, a \, \cdots \, b^{i_n-i_{n-1}} \, a \, b^{-i_n} \in \mathrm{A} \ast \mathrm{B}$. Counting the occurrences of $b^{\pm 1}$ gives
    \[
        \ell_T(h)=2\big[n-1+\1_{\{i_1\neq0\}}+\1_{\{i_n\neq0\}}\big].
    \]
    Thus, $m_z|_H$ is a radial phase with a correction depending on the first and last letters, rather than an exactly radial symbol in general.

    The Cayley graph of $H$ for these generators is the $3$-regular tree. Its radial algebra $\A_\rad$ is again generated by the sum of unit length elements $\chi_1$ and therefore Abelian. The spectral transform of $\chi_1/(2\sqrt2)$ identifies $\A_\rad$ with $L_\infty([-1,1],\nu_{2/3})$ and sends the normalized sphere sums $\chi_n/\|\chi_n\|_2$ to the orthonormal polynomials $Q_n$ considered in \eqref{eq:Qrecurrence} with parameter $\tau = \tfrac{2}{3}$, we refer to \cite{BoFen2006} for further details. Although $T_{m_z}$ need not preserve this algebra, its radial compression $\EE_{\A_\rad} T_{m_z}|_{\A_\rad}$ is represented on radial polynomials, with $w=z^2$, by
    \[
        \EE_{\A_\rad} T_{m_z}|_{\A_\rad} \ \mapsto \  \mathscr{A}_z
        =\frac{(1+2w)^2}{9w}\mathcal J_w
        -\frac{4(w-1)^2}{9w}\mathcal J_{-\frac{w}{2}}
        +\frac{w^{-1}-1}{3}\Delta_0,
    \]
    where $\mathcal{J}_w Q_n=w^n Q_n$ and $\Delta_0 f= \nu_{2/3}(f) \, \1$. Since the conditional expectation is contractive, boundedness of $T_{m_z}$ would imply the boundedness of the compressed operator $\mathscr{A}_z$. 
    The last two terms are bounded on every $L_p(\nu_{2/3})$, whereas the coefficient of $\mathcal{J}_w$ is nonzero. Thus, the generating-function argument of Lemma~\ref{prop:generating} with $\tau=\tfrac{2}{3}$ makes the compression---as well as the multiplier $T_{m_z}$---unbounded whenever $z^4\neq1$ and $p\neq2$.

    For the converse $m_{\pm1}\equiv1$, while $m_{\pm i}$ is the product of the character $\chi(a)=-1$, $\chi(b)=1$ and a symbol determined by the first and last free factors in $\mathrm{A} \ast \mathrm{B}$. Thus, complete boundedness follows from \cite[Theorem~3.5]{MeiRicard2017FreeHilb}. At $p=1,\infty$, boundedness holds precisely when $z^2=1$.
\end{remark}

\subsection{A cocycle weak law}

All that remains to do is to prove the weak law in Proposition \ref{prp:cocycleWeakLaw} for the modular cocycle. Before proceeding with the next lemma, it would be convenient to review the following well-known fact about continued fractions. First, given an irrational number $x \in (0,1)$, its continuous fraction representation is the unique expression 
\[
    x = [0;d_1, d_2, d_3, \dots ] := \cfrac{1}{d_1 + \cfrac{1}{d_2 + \cfrac{1}{d_3 + \cdots}}}.
\]
Observe that the first digit and the left-shifted continued fraction are easily obtained by taking integer parts as follows 
\[
    \G(x) = \left\{ \frac1{x} \right\} = \frac1x-\left\lfloor\frac1x\right\rfloor,
    \qquad x_j=\mathcal G^j(x),\qquad
    d_j = \left\lfloor\frac1{x_{j-1}}\right\rfloor
    \quad \text{ for every } \, j \geq 1.
\]
The map $\G \colon (0,1) \to (0,1)$, sometimes referred to as the Gauss map, is measurable and well defined in a set of total Lebesgue measure. Although $\G$ is not invertible, since it loses the information of the first digit $d_1$, it can easily be made so by extending it to the map $$\widehat{\G}: (0,1)^2 \to (0,1)^2$$ that pushes $d_1$ from one continuous fraction to the other
\[
    \widehat{\G}(x,v) = \Big( \frac1{x} - d_1, \, \frac1{d_1 + v} \Big), \quad \text{ where } \; d_1 = \left\lfloor \frac1{x} \right\rfloor.
\]
It is not difficult to check that  $\widehat{\G}$ preserves the probability measure 
\[
d \widehat{\eta}(x,v) = \frac{1}{\log 2}\frac{dx dv}{(1 + x v)^{2}}.
\]
Thus, the Gauss map $\G$ preserves its $x$-marginal $d\eta$, which is given by 
\[
    d\eta(x) = \frac1{\log 2} \frac{d x}{1 + x}.
\]
The fact that shifting continued fraction expansions preserves this measure was already known to Gauss, hence the name. We refer to  \cite{KatokUgarcovici2007} for the historical remarks. The measure $\widehat{\eta}$ can be made into a rescaled Lebesgue measure under the change of variables $\Phi: (0,1)^2 \to \Sigma_0$ which is given by
\[
    \Phi(x,v) = \Big( x, \, \frac{v}{1 + xv} \Big) \quad \text{and} \quad \Sigma_0 = \Big\{ (x,y): x \in (0,1), \; 0<y<\frac1{1+x} \Big\}.
\]
This transforms $\widehat{\G}$ into $\R_0: \Sigma_0 \to \Sigma_0$ given by
\[
    \R_0(x,y) := \big( \G(x), x - x^2 y \big).
\]
Now, we will employ another realization of the fundamental domain of the lattice $\Gamma \leq \sG$ in terms of continued fractions. We start by introducing the following extensions of the domain $\Sigma_0$ above
\begin{align*}
    \Sigma & = \Big\{ (x,y,\varepsilon): (x,y) \in \Sigma_0, \; \varepsilon\in\{0,1\}\Big\},\\
    \widetilde{\Sigma} & = \Big\{(x,y,\varepsilon,u): (x,y,\varepsilon) \in\Sigma, \; 0 \le u < -2\log x \Big\}.
\end{align*}
Take also
\[
    s_0(x,y)=\begin{bmatrix}1&y\\-x&1-xy\end{bmatrix} \quad \text{and} \quad
    s_1(x,y)=\begin{bmatrix}x&1-xy\\-1&y\end{bmatrix}.          
\]
The following lemma is from \cite{Arnoux1994}. As before, we will again normalize the Haar measure of $\sG$ so that $\Gamma \backslash \sG$ has an invariant probability measure $\mu$ that is scale compatible with the Haar measure.

\begin{lemma} \label{lem:arnoux-fd}
The following properties hold:
\begin{enumerate}[label={\rm \textbf{(\arabic*)}}, ref={\rm (\arabic*)}]
	    \item \label{lem:arnoux-fd1} The subset $\{  s_\varepsilon(x,y)a_u : (x,y,\varepsilon,u)\in\widetilde{\Sigma}\} \subset \sG$ is a measurable fundamental domain for $\Gamma \backslash \sG$.
        \item \label{lem:arnoux-fd2} The associated map $z \colon \widetilde{\Sigma} \to \Gamma \backslash \sG$ given by $z=z(x,y,\varepsilon,u)=\Gamma s_\varepsilon(x,y)a_u$ is bijective modulo null sets. Moreover, taking $\nu = z_*^{-1} \mu$ and writing $\#$ for the counting measure, we have that
        \[
            d\nu (x,y,\varepsilon,u)=\frac{3}{\pi^2} \, dx\,dy\,d \#(\varepsilon) \, du.
	    \]
        \item \label{lem:arnoux-fd3} The map $\R:\Sigma \to \Sigma$ given by $\R(x,y,\varepsilon) = \left(\R_0(x,y), 1-\varepsilon\right)$ preserves the measure $dx\,dy\,d\#(\varepsilon)$. 
	\end{enumerate}
\end{lemma}
\begin{remark} \normalfont
\label{rmk:ReturnTime}
The map $\R$ on the last item admits a very natural geometric interpretation. Indeed, if we consider the embedding $\iota \colon \Sigma \to \Gamma \backslash \sG$, its image $\Cc = \iota[\Sigma]$ is a two-dimensional subset composed of two sheets. Now, take the right action of the one-parameter group $t \mapsto a_t$ inside $\sG$ defined in \eqref{eq:defAgroup}. Clearly, it induces a flow $\phi_t \colon \Gamma \backslash \sG \to \Gamma \backslash \sG$ given by $\phi_t(\Gamma g) = \Gamma g a_t$. If $\zeta = (x,y, \varepsilon) \in \Sigma$ are the coordinates of a point in $\iota(\zeta) \in \Cc$, the map $\R$ describes the coordinates on $\Cc$ of the first point of the flow $t \mapsto \phi_t(\iota(\zeta))$ that returns to $\Cc$, while the return time is given by $r(x) = -2 \log x$. In other words, we have the identity $\phi_{r(x)}(\iota(\zeta)) = \iota(\R(\zeta))$. This justifies why the parametrization of $\Do$ contains $\Cc$ and an extra variable in $[0,-2\log x]$.
\end{remark}

Let $(x_0,y_0,\varepsilon_0)=(x,y,\varepsilon)$, we consider the iterates of the transform $\R$ for $j \ge 1$
\begin{equation} \label{Eq-Iteration}
    (x_{j},y_{j},\varepsilon_{j})=\R^{j}(x,y,\varepsilon) = \Big( 
    \G(x_{j-1}), \, x_{j-1}-x_{j-1}^2y_{j-1}, \, \varepsilon+j \!\! \pmod2 \Big).
\end{equation} 
Set $\tau_j=-2\log x_j$, where $j \geq 0$, to be the time it takes to reenter $\Cc$ after the $j^{\text{th}}$ iteration. Then given $z = z(x,y,\varepsilon,u) \in \Gamma \backslash \sG$ and \(t>0\), we can define \(n_t(z) \in \NN\) as the stopping time defined as follows
\[
    n_t(z) = \min \Big\{ n \geq 0 \mid S_n > t + u \Big\}, \quad \text{ where } \quad S_n := \sum_{j = 0}^{n} \tau_j.
\]
Here $u$ is the last coordinate of $z$. Clearly $S_{n_t(z)-1} \leq u+t < S_{n_t(z)}$ and $n_t(z)$ is well-defined a.e.

Let $\sigma_0 \colon \Gamma \backslash \sG \to \widetilde \Sigma \subset \sG$ be the section given by the inverse of the parametrization in Lemma \ref{lem:arnoux-fd} \ref{lem:arnoux-fd2}, and let \(\beta_0 \colon \Gamma \backslash \sG \times \sG\to \Gamma\) denote the associated right cocycle satisfying $\sigma_0(x)g=\beta_0(x,g)\sigma_0(xg)$ a.e. We derive the following useful formula for $\beta_0$. For \(d\in\NN\), write
\[
W_0(d)=(b^2a)^d \quad \text{and} \quad W_1(d)=(ba)^d.   
\]

In what follows, we shall work with functions defined on $z \in \Gamma \backslash \sG$ where $z = z(x,y,\varepsilon, u)$ and $(x,y,\varepsilon, u) \in \widetilde{\Sigma}$ as in Lemma \ref{lem:arnoux-fd} \ref{lem:arnoux-fd2}. Moreover, as above, the variable $x \in (0,1) \setminus \QQ$ admits the expression $x = [0;d_1, d_2, d_3, \dots]$ as a continued fraction. We have seen that the assignment function $x \mapsto d_j(x)$ satisfies $d_j(x) = \lfloor 1/x_{j-1} \rfloor$ where $x_{j-1}$ is given by \eqref{Eq-Iteration}. Assuming that $(x,y,\varepsilon,u) \in \widetilde{\Sigma}$ is given, we shall write $d_j$ for notational simplicity to refer to the evaluation $d_j(x)$.  

\begin{lemma}\label{lem:return-cocycle}
	The following identities hold a.e. 
	\begin{equation*} 
		\beta_0(z,a_t)
		=
		\prod_{1 \le j \le n_t(z)}^{\rightarrow}
		W_{\varepsilon+j-1\ ({\rm mod}\ 2)}
		\left( d_j \right) 
        \qquad \text{and} \qquad
        \ell_T(\beta_0(z,a_t))
		=
		2\sum_{j=1}^{n_t(z)} d_j,
	\end{equation*}
	The product arrow $\rightarrow$ means that the factors are ordered from left to right in increasing order of \(j\).
\end{lemma}

\begin{proof}
	Note that 
	\[
	\begin{bmatrix}
		1& d_{j+1}\\
		0&1
	\end{bmatrix}^{-1}
	=
	(b^2a)^{d_{j+1}}
	=
	W_0(d_{j+1}),\quad 	
    \begin{bmatrix}
		1 & 0\\
		d_{j+1} &1
	\end{bmatrix}^{-1}
	=
	(ba)^{d_{j+1}}
	=
	W_1\left(d_{j+1}\right).
	\]
	A straightforward matrix calculation yields that
	\begin{equation*} 
		s_{\varepsilon_j}(x_j,y_j)\,a_{\tau_j}
		=
		W_{\varepsilon_{j}}
		\left(d_{j+1}\right)
		s_{\varepsilon_{j+1}}(x_{j+1},y_{j+1})
        \quad \text{ for every } j \geq 0.
	\end{equation*}
	By iteration, we obtain that 
	\[
		s_\varepsilon(x,y)\,a_{S_{n}}
		=
		\bigg[
		\prod_{1 \le j \le n}^{\rightarrow}
		W_{\varepsilon+j-1\ ({\rm mod}\ 2)}
		\left(
		d_j
		\right)
		\bigg]
		s_{\varepsilon_n}(x_n,y_n),
    \]
    where again $S_{n} = \sum_{j=0}^{n-1}\tau_j$. In particular
	\begin{equation}
        \label{eq:iterated-return-cocycle} 
        s_\varepsilon(x,y)a_{u+t}
        = 
        \bigg[
            \prod_{1 \le j \le n}^{\rightarrow}
            W_{\varepsilon+j-1\ ({\rm mod}\ 2)}
            \left(d_j\right)
	    \bigg]
        s_{\varepsilon_n}(x_n,y_n) a_{u+t-S_{n}}. 	
    \end{equation}
	By definition, we have $0 \leq  u+t-\sum_{j=0}^{n-1}\tau_j < \tau_n$ for $n=n_t(z)$ and hence the element after the bracketed product in \eqref{eq:iterated-return-cocycle} belongs to the fundamental domain in Lemma \ref{lem:arnoux-fd},
	which in particular yields that  
	\begin{equation}
        \label{eq:representative-after-flow}
		\sigma_0(za_t)
		= \sigma_0(\Gamma s_\varepsilon(x,y)a_{u+t} ) =
		s_{\varepsilon_n}(x_n,y_n)
		a_{u+t-S_{n}}.
	\end{equation}
	Note that $\sigma_0(z)a_t
	=
	s_\varepsilon(x,y)a_{u+t}$. So together with 
	\eqref{eq:iterated-return-cocycle} and definition of the cocycle $\beta_0$, we obtain the first identity of the lemma.
	Moreover, the last assertion immediately follows from identity \eqref{Eq-nB} 
	\[
	\ell_T(\beta_0(z,a_t))
	=
	2n_\mathrm{B}(\beta_0(z,a_t))
	=
	2\sum_{j=1}^{n_t(z)}
	d_j. \qedhere
	\]
\end{proof}

We use Khintchine’s weak law for the averages of continued fractions, see \cite[Section 4, p.~377]{Khintchine1935}.
\begin{lemma}\label{lem:khintchine}
	The functions $d_j = d_j(x)$ satisfy that
	\begin{equation}\label{eq:khintchine-weak-law}
		\frac{1}{n\log n}
		\sum_{j=1}^{n}
		d_j
        \,
		\longrightarrow
        \, 
        \frac1{\log2}
	\end{equation}
	in probability as $n \to \infty$ with respect to the Lebesgue measure on the unit interval $(0,1)$.
\end{lemma}

Finally, the remaining weak law for the cocycle length is in reach.
\begin{proof}[Proof of Proposition~\ref{prp:cocycleWeakLaw}]
We claim that 
\[
\varphi_n(x) := \frac{1}{n} \sum_{j=0}^{n-1} \underbrace{\log \Big( \frac{1}{x_j^2} \Big)}_{\tau_j} \longrightarrow \frac{\pi^2}{6 \log 2} \quad \text{a.e. } x \in (0,1). 
\]
In fact, take $f(x,y, \varepsilon) = -2 \log(x)$. According to Lemma \ref{lem:arnoux-fd} \ref{lem:arnoux-fd3}, $\R \colon \Sigma \to \Sigma$ is measure preserving when $\Sigma$ is equipped with the probability measure $d\sigma(x,y,\varepsilon) = \tfrac{1}{2 \log 2} dx \, dy \, d \# (\varepsilon)$. Since $\R \colon \Sigma \to \Sigma$ is ergodic \cite[\S 5-6]{Arnoux1994}, Birkhoff's ergodic theorem applies to $f \circ \R^j$ and
\[
    \frac1n\sum_{j=0}^{n-1}\tau_j
    \longrightarrow
    -2 \int_\Sigma \log x \, \frac{dx \, dy \, d \# (\varepsilon)}{2\log 2} =
    \frac{\pi^2}{6\log2},\quad  a.e.\  (x,y,\varepsilon) \in \Sigma.
\]
By the definition of \(n_t(z)\), we have \(n_t(z)\to\infty\) as \(t\to\infty\) and
	\[
	\frac{1}{n_t(z)}
	\sum_{j=0}^{n_t(z)-1}\tau_j
	<
	\frac{u+t}{n_t(z)}
	<
	\frac{1}{n_t(z)}
	\sum_{j=0}^{n_t(z)}\tau_j.
	\]
Taking the limit and combining it with Lemma
\ref{lem:khintchine} yields
\[
        \frac{n_t(z)}{t} \longrightarrow  		\frac{6\log2}{\pi^2} \qquad \text{and} \qquad \frac{1}{t\log t}
		\sum_{j=1}^{n_t(z)} d_j
		\longrightarrow
		\frac6{\pi^2}
\]
in measure. Lemma \ref{lem:return-cocycle} now yields the first assertion for the cocycle $\beta_0$. To pass from $\beta_0$ to a general $\beta$, let $\sigma \colon \Gamma \backslash \sG \to \sG$ be the section associated with $\beta$ by $\sigma(x)g=\beta(x,g)\sigma(xg)$. Write \(\sigma(z)=q(z)\sigma_0(z)\) for \(q(z)\in\Gamma\) so that $\beta(z,g)= q(z) \beta_0(z,g) q(zg)^{-1}$.
Hence
\[
	\big|
	\ell_T(\beta(z,a_t))
	-
	\ell_T(\beta_0(z,a_t))
	\big|
	\leq
	\ell_T(q(z))+\ell_T(q(za_t)).
\]
Since \(\mu\) is invariant under \(z\mapsto za_t\), we obtain 
	\[
	\mu\big\{
		z:
		\ell_T(q(za_t))>\eta t\log t
		\big\}
		=
		\mu\big\{
		z:
		\ell_T(q(z))>\eta t\log t
		\big\}
	\to 0,\quad \eta>0.
	\]
	Hence
	\[
	\frac{
		\ell_T(\beta(\cdot ,a_t))-\ell_T(\beta_0(\cdot,a_t))
	}{t\log t}
	\longrightarrow0
	\]
in probability. Together with the first assertion for $\beta_0$, this proves the desired convergence for the cocycle $\beta$. To conclude, let $R_j\to\infty$, $z_j = \exp ( i \pi^2/12R_j\log R_j)$ and fix \(s>0\). According to the first assertion and since $\log(R_js)/\log R_j \to 1$,
	we have
	\[
	m_{z_j}(\beta(\cdot,a_{R_js}))
	=
	\exp\left( 	i\frac{\pi^2\ell_T(\beta(\cdot,a_{R_js}))}
	{12R_j\log R_j}
	\right)
	\longrightarrow e^{i s}
	\]
	in \(\mu\)-probability.	
	Now, since
	\( m_{z_j}(\beta(\cdot,a_{R_js})) \) are uniformly bounded 
	\[
    \lim_{j \to \infty} \widehat m_{z_j}(a_{R_js})
		 = \lim_{j \to \infty}
			\int_{\mathrm{\Omega}}
			m_{z_j}(\beta(x,a_{R_js}))\,d\mu(x) =  e^{i s}. \qedhere
	\]
	\end{proof}

\section{Some open problems}\label{Sect:Problems}

\subsection{Nonamenable groups and Fourier-Schur transference} 
Let $\sG$ be a l.c.s.c. group and $p \in (1,\infty)\setminus \{2\}$. We will say that the group has Fourier-Schur transference in $L_p$ if the following inequality holds
\begin{equation}
    \label{eq:TransferenceProperty}
    \big\| T_m: L_p(\L \Gamma) \to L_p(\L \Gamma) \big\|_\cb \leq C \, \big\| H_m: S_p(\Gamma) \to S_p(\Gamma) \big\|_\cb,
\end{equation}
for every bounded function $m \in L_\infty(\sG)$ and a constant $C > 0$ independent of $m$. As usual, we will use the convention that the quantities above are infinite when the function $m$ does not induce a completely bounded Fourier or Herz-Schur multiplier. Observe also that, since the reverse inequality holds contractively, see \cite[Theorem 4.2]{CasSall2015}, the identity map on $m$ induces a bounded contractive inclusion from Fourier to Herz-Schur multipliers. Thus, the inequality \eqref{eq:TransferenceProperty} is equivalent, due to the closed graph theorem, to the equality as sets of Fourier and Herz-Schur multipliers in $L_p$ and $S_p$ respectively.

Now that we know that no discrete group containing a free subgroup has the Fourier-Schur transference property, it is natural to ask whether that property, either for some $p$ or for all the nontrivial range, is enough to ensure amenability. 

\begin{problem}
    \label{prb:reciprocal}
    Fix $1 < p \neq 2 < \infty$. Is a group $\Gamma$ satisfying \eqref{eq:TransferenceProperty} automatically amenable?
\end{problem}

The closest result in this direction is in \cite[Theorem 2.1]{CasSall2015}, in which Caspers/de la Salle showed that the technique employed to prove Fourier-Schur transference in the amenable case does not work in discrete nonamenable groups. However, their result uses as an hypothesis the existence of an intertwining isometry from $L_p(\L \sG)$ into an ultraproduct of Schatten classes and is not clear at present if a contradiction would also follow from just assuming a quasi-isometric embedding. A solution in either direction to the above problem would be of great interest.

Observe as well that all of our examples ---free groups, $\PGL_2(\QQ_q)$ and $\PSL_2(\RR)$--- exhibit \textit{rank one behavior}. An extremely interesting question is whether it is possible to produce natural $\sK$-biinvariant Herz-Schur multipliers in $S_p$ that are unbounded as Fourier multipliers in $L_p$ for higher rank groups. Perhaps the simplest case would be that of $\sG = \PGL_3(\QQ_q)$ with its Bruhat-Tits' building playing the role of the homogeneous tree $X_{q+1}$ in the rank one case. Nevertheless, in this simple setting, the failure of the Schur approximation property for large $p$ \cite{LaffSall2011} already imposes a restriction on Schur boundedness that is absent in the case of $\PGL_2(\QQ_q)$. Indeed, the analogues of the multipliers on Theorem \ref{corFree2}, that is, two-parameter phases associated to the dominant coweights of the Bruhat-Tits building, cannot be uniformly bounded as Schur multipliers.

\subsection{Growth of the Fourier-Schur ratio} \label{SectFSTransfRatio} 
Since now we have counterexamples to Fourier-Schur transference, the local transference results studied in \cite{ParcetSLn2018} become of renewed interest. Indeed, consider the ratio
\[
    \R_{\sG,p}(\Omega) := \sup_{\supp m \hskip1pt \subset \hskip1pt \Omega} \frac{\| T_m \colon L_p(\L \sG) \to L_p(\L \sG)\|_\cb}{\| S_M \colon S_p(\sG) \to S_p(\sG) \|_\cb}
\]
for any compact set $\Omega$ in a given l.c. group $\sG$. Here $1 \le p \le \infty$ and $M(g,h) = m(g^{-1}h)$ is the Herz-Schur symbol associated to $m$. Theorem \ref{CorFree}\ref{itm:CorFree1} implies that $\varphi_p(n) = \R_{\FF_r,p}(B_n(e))$ diverges to $\infty$ with $n$, but its growth rate  remains open. This has been investigated in \cite{ParcetSLn2018, Parcet2024Idempotentschur} in connection with the amenability modulus of a compact set $\Omega$
\[
    \delta_\sG(\Omega) := \inf_{\| \xi \|_{L_2(\sG)} = 1} \sup_{g \in \Omega} \big| 1 - \langle \xi, \lambda_g \xi \rangle \big|.
\]
By \cite[Lemma 1.1]{ParcetSLn2018}, $\delta_{\sG}(\Omega) = 0$ for every compact set $\Omega$ when $\sG$ is amenable, while $\delta_{\sG}(\Omega_n) \to 1$ as the compact sets $\Omega_n$ increase to fill the whole group when $\sG$ is nonamenable. In \cite[Section 1]{ParcetSLn2018} the authors obtained that 
\[
    \R_{\sG,2m}(\Omega) \leq \frac{1}{1 - \delta_{\sG} \big( (\Omega \Omega^{-1})^{m} \big)}. 
\]
Although refined inequalities were found in \cite[Section 3]{Parcet2024Idempotentschur}, all the known inequalities are far from optimal. Indeed, as $p \to 2,\infty$ the constant $\R_{\sG,p}(\Omega)$ should go to $1$, while the known estimates always diverges when $\Omega$ contains a generating set. Similar arguments have been used in the context of Fourier restriction \cite{CaspersJanssensKrishnaswamyUshaMiaskiwskyi2024}. This  raises the following natural question for nonamenable groups.

\begin{problem}
     What is the optimal dependency of $\R_{\sG,p}(\Omega)$ in terms of $p$ and $\delta_\sG(\Omega)$?
\end{problem}

In personal communications, Patrick Poissel shared with the authors that he showed that the above dependency can be improved in a way that does not diverge as $p \to  \infty$. However, his approach comes at the expense of using a different amenability modulus obtained by quantifying the local existence of Fourier-Stieltjes approximations of the identity in the multiplier algebra.

\subsection{Counterexamples to de Leeuw restriction}

Let $\Gamma \leq \sG$ be a closed discrete subgroup of the l.c.s.c group $\sG$. Determining under which conditions it is possible to restrict a (continuous) $L_p$-bounded Fourier multiplier from $\sG$ to $\Gamma$ is currently open. In the case of $\RR^n$ and other locally compact Abelian groups it is always possible to do so \cite{deLeeuw1961}, but the question is still open in general. If Fourier-Schur transference were to hold for the discrete group $\Gamma$, Lafforgue/de la Salle sampling theorem \cite[Theorem 1.19]{LaffSall2011} would imply the restriction estimate
\[
    \big\| T_{m|_\Gamma}: L_p(\L \Gamma) \to L_p(\L \Gamma) \big\|_\cb 
    \; \leq \; C \, \big\| T_{m}: L_p(\L \sG) \to L_p(\L \sG) \big\|_\cb,
\]
for every $m$ continuous. In the absence of Fourier-Schur transference, Caspers, Parcet, Perrin and Ricard \cite{CasParPerrRic2014} showed that Fourier multiplier restriction holds when $\sG$ has the \emph{small almost $\Gamma$-invariant} neighborhood property or $[\mathrm{SAIN}]_\Gamma$, see also \cite{CaspersJanssensKrishnaswamyUshaMiaskiwskyi2024} for local and multilinear extensions. However, this property fails for natural pairs like $\PSL_2(\ZZ) \leq \PSL_2(\RR)$. The reason why the $[\mathrm{SAIN}]_\Gamma$ condition is required in their proof mimics why amenability was required in \cite{NeuRic2011}. Indeed, $[\mathrm{SAIN}]_\Gamma$ is needed to overcome the fact that the natural embedding $\L \Gamma \into \L \sG$ is not trace-preserving if $\sG$ is nondiscrete. All of this strongly suggests that counterexamples to Fourier restriction along the lines of the Fourier-Schur transference should exist. 

 
\textbf{AI-use statement.} Inequality (SphT) was produced by ChatGPT-5.6 in late June, in response to questions concerning previous results from a draft paper—already mathematically complete at the time—on mixed norms for free-group multipliers \cite{MixedNorm2026}, where the key martingale filtration on trees was used in terms of the hyperbolic boundary of a free group. After realizing the importance of the technique around Lemma \ref{KeyLemma}---relying on both its simplicity and its depth---this text slowly grew as an effort to understand its full scope, implications and also limitations for Fourier and Schur multipliers. We hope to have made these points clear around this intriguing problem. During this process, AI tools were used to aid with aside computations and references. All mathematical arguments in this manuscript were verified and written by the authors.


\textbf{Acknowledgement.} We want to express our gratitude to Tao Mei, \'Eric Ricard, Mikael de la Salle and Hong-Wei Zhang for several helpful comments. A. González-Pérez was supported by the Ramón y Cajal grant \texttt{RYC2022-037045-I} (Ministry of Science, Spain) while in a research stay at Xiamen University. J. Parcet was supported by the European Research Council: Advanced Grant SMARP \texttt{101263125} and  Severo Ochoa Grant \texttt{CEX2023-001347-S}. The first three-named authors were also supported in part by Spanish Grant \texttt{PID2022-141354NB-I00}. J.P. García and S. Wang were supported by the NSF of China: \texttt{No.12671160}, \texttt{No.W2611005}, \texttt{No.W2441002}, \texttt{No.12031004}.  

\printbibliography

\vskip25pt

\small

\textbf{Adrián  M. González-Pérez} \hfill \textbf{Javier Parcet} \\ 
Departamento de Matemáticas \hfill Instituto de Ciencias Matem{\'a}ticas
\\ Universidad Aut\'onoma de Madrid \hfill Consejo Superior de Investigaciones Científicas
\\ \texttt{adrian.gonzalez@uam.es} \hfill \texttt{parcet@icmat.es}

\vskip5pt

\textbf{Jorge Pérez García} \hfill \textbf{Simeng Wang} \\
IAS Mathematics \hfill IAS Mathematics 
\\ Harbin Institute of Technology \hfill Harbin Institute of Technology  
\\ \texttt{moxim@hotmail.es} \hfill \texttt{simeng.wang@hit.edu.cn}

\end{document}